\documentclass{scrartcl}
\usepackage{amssymb}
\usepackage{amsfonts}
\usepackage{amsmath}
\usepackage{hyperref}
\usepackage{pgf}
\usepackage{tikz}
\usetikzlibrary{decorations.pathreplacing,patterns, hobby, shapes.geometric, 3d,calc, arrows}
\usepackage{subfigure}
\usepackage{nicefrac}
\newcommand{\bbbc}{\mathbb{C}}
\newcommand{\bbbn}{\mathbb{N}}
\newcommand{\bbbr}{\mathbb{R}}
\newcommand{\Idx}{\mathcal{I}}

\newcommand{\Kdx}{\mathcal{K}}
\newcommand{\ctI}{\mathcal{T}_{\Idx}}
\newcommand{\lfI}{\mathcal{L}_{\Idx}}
\newcommand{\ctIl}[1]{\mathcal{T}_{\Idx}^{(#1)}}

\newcommand{\ctII}{\mathcal{T}_{\Idx\times\Idx}}
\newcommand{\lfII}{\mathcal{L}_{\Idx\times\Idx}}
\newcommand{\lfiII}{\mathcal{L}_{\Idx\times\Idx}^-}
\newcommand{\lfaII}{\mathcal{L}_{\Idx\times\Idx}^+}
\newcommand{\dist}{\mathop{\operatorname{dist}}\nolimits}
\newcommand{\diam}{\mathop{\operatorname{diam}}\nolimits}
\newcommand{\sons}{\mathop{\operatorname{chil}\nolimits}}
\newcommand{\supp}{\mathop{\operatorname{supp}}\nolimits}
\newcommand{\level}{\mathop{\operatorname{level}}\nolimits}
\newcommand{\brow}{\mathop{\operatorname{row}}\nolimits}

\newcommand{\anc}{\mathop{\operatorname{anc}}\nolimits}
\newcommand{\sd}[1]{\mathop{\operatorname{sd}}_{#1}\nolimits}

\newcommand{\bi}{\mathbf{i}}
\newcommand{\pI}{p_{\Idx}}

\newcommand{\Cdi}{C_\text{di}}
\newcommand{\Clv}{C_\text{lv}}
\newcommand{\Csb}{C_\text{cp}}
\newcommand{\Csp}{C_\text{sp}}
\newcommand{\Cspt}{C_{\text{sp},t}}
\newcommand{\Csn}{C_\text{nc}}
\newcommand{\Cbp}{C_\text{cu}}
\newcommand{\Cbb}{C_\text{cl}}
\newcommand{\Cov}{C_\text{ov}}
\newcommand{\Clf}{C_\text{lf}}
\newcommand{\Crs}{C_\text{rs}}
\newcommand{\Cun}{C_\text{un}}

\newcommand{\Cbs}{C_\text{bs}}
\newcommand{\Ccs}{C_\text{cs}}
\newcommand{\Cqr}{C_\text{qr}}
\newcommand{\Csvd}{C_\text{svd}}
\newcommand{\Cbw}{C_\text{bw}}
\newcommand{\Cwe}{C_\text{we}}
\newcommand{\Ctr}{C_\text{tr}}
\newcommand{\Cpr}{C_\text{pr}}

\newtheorem{theorem}{Theorem}
\newtheorem{lemma}[theorem]{Lemma}
\newtheorem{definition}[theorem]{Definition}
\newtheorem{remark}[theorem]{Remark}

\newenvironment{proof}{\emph{Proof.}}{\hfill$\Box$}

\allowdisplaybreaks

\title{Hybrid matrix compression for high-frequency problems}
\author{Steffen B\"orm and Christina B\"orst}
\date{\today}
\begin{document}
\maketitle

\begin{abstract}
Boundary element methods for the Helmholtz equation lead to large
dense matrices that can only be handled if efficient compression
techniques are used.
Directional compression techniques can reach good compression rates
even for high-frequency problems.

Currently there are two approaches to directional compression:
analytic methods approximate the kernel function, while algebraic
methods approximate submatrices.
Analytic methods are quite fast and proven to be robust, while
algebraic methods yield significantly better compression rates.

We present a hybrid method that combines the speed and reliability
of analytic methods with the good compression rates of algebraic
methods.
\end{abstract}

%
%
\section{Introduction}

We consider the Helmholtz single layer potential operator
\begin{equation*}
  \mathcal{G}[u](x) := \int_\Omega g(x,y) u(y) \,dy,
\end{equation*}
where $\Omega\subseteq\bbbr^3$ is a surface and
\begin{equation}\label{eq:helmholtz}
  g(x,y) = \frac{\exp(\bi \kappa \|x-y\|)}{4\pi \|x-y\|}
\end{equation}
denotes the Helmholtz kernel function with the
wave number $\kappa\in\bbbr_{\geq 0}$.

Applying a standard Galerkin discretization scheme with
a finite element basis $(\varphi_i)_{i\in\Idx}$ leads to the
stiffness matrix $G\in\bbbc^{\Idx\times\Idx}$ given by
\begin{align}\label{eq:matrix}
  g_{ij}
  &= \int_\Omega \varphi_i(x) \int_\Omega g(x,y) \varphi_j(y) \,dy\,dx &
  &\text{ for all } i,j\in\Idx,
\end{align}
where we assume that the basis functions are sufficiently smooth
to ensure that the integrals are well-defined even if the supports
overlap, e.g., for $i=j$.
Due to $g(x,y)\neq 0$ for all $x\neq y$, the matrix $G$ is not
sparse and therefore requires special handling if we want to
construct an efficient algorithm.

Standard techniques like fast multipole expansions \cite{RO85,GRRO87},
panel clustering \cite{HANO89,SA00}, or hierarchical matrices
\cite{HA99,HAKH00,GRHA02} rely on local low-rank approximations
of the matrix.
In the case of the high-frequency Helmholtz equation, e.g., if
the product of the wave number $\kappa$ and the mesh width $h$
is bounded, but not particularly small, these techniques
can no longer be applied since the local ranks become too large.
This situation frequently appears in engineering applications.

The \emph{fast multipole method} can be generalized to handle this
problem by employing a special expansion that leads to operators
that can be diagonalized, and therefore evaluated efficiently
\cite{RO93,GRHUROWA98}.

The \emph{butterfly method} (also known as multi-level matrix
decomposition algorithms, MLMDA) \cite{BOMI96} achieves a similar
goal by using permutations and block-diagonal transformations in
a pattern closely related to the fast Fourier transformation
algorithm.

\emph{Directional methods} \cite{BR91,ENYI07,MESCDA12,BEKUVE15}
take advantage of the fact that the Helmholtz kernel (\ref{eq:helmholtz})
can be written as a product of a plane wave and a function that
is smooth inside a conical domain.
Replacing this smooth function by a suitable approximation results
in fast summation schemes.

We will focus on directional methods, since they can be applied
in a more general setting than the fast multipole expansions based
on special functions, and since they offer the chance of achieving
better compression to $\mathcal{O}(n \log n)$ coefficients
compared to $\mathcal{O}(n \log^2 n)$ required by the
butterfly scheme \cite{BOMI96}.

In particular, we will work with \emph{directional $\mathcal{H}^2$-matrices}
(abbreviated $\mathcal{DH}^2$-ma\-tri\-ces), the algebraic counterparts of the
directional approximation schemes used in \cite{BR91,ENYI07,MESCDA12}.
Our starting point is the directional Chebyshev approximation scheme
introduced in \cite{MESCDA12} and analyzed in \cite{BOME15}.
While this approach is fast and proven to be reliable, the resulting
ranks are quite large, and this leads to unattractive storage
requirements.

We can solve this problem by applying an algebraic recompression
algorithm that starts with the $\mathcal{DH}^2$-matrix constructed
by interpolation and uses nested orthogonal projections and singular
value decompositions (SVD) to significantly reduce the rank.
This algorithm is based on the general $\mathcal{DH}^2$-matrix
compression algorithm introduced in \cite{BO15}, but takes advantage
of the previous approximation in order to significantly reduce the
computational work to $\mathcal{O}(n k^3 \log n)$ in
the high-frequency case, cf. Theorem~\ref{th:complexity} and
Remark~\ref{re:complexity}.

Compared to the closely related algorithm presented in \cite{MESCDA12},
our algorithm compresses the entire $\mathcal{DH}^2$-matrix structure
instead of just the coupling matrices, and the orthogonal projections
applied in the recompression algorithm allow us to obtain straightforward
estimates for the compression error.

Compared to the algorithm presented in \cite{BEKUVE15}, our approach
has better stability properties, owing to the results of \cite{BOME15}
for the interpolation scheme and the orthogonal projections employed
in \cite{BO15} for the recompression, and it can be expected to yield
better compression rates, since it uses an $\mathcal{H}^2$-matrix
representation for low-frequency clusters, while the algorithm of
\cite{BEKUVE15} relies on the slightly less efficient $\mathcal{H}$-matrices.

%
%
\section{\texorpdfstring{Directional $\mathcal{H}^2$-matrices}
                        {Directional H2-matrices}}

Hierarchical matrix methods are based on decompositions of the
matrix $G$ into submatrices that can be approximated by factorized
low-rank matrices.
In our case, we follow the directional interpolation technique
described in \cite{MESCDA12} and translate the resulting compressed
representation into an algebraical definition that can be applied
in more general situations.

In order to describe the decomposition into submatrices, we first
introduce a hierarchy of subsets of the index set $\Idx$ corresponding
to the \emph{box trees} used, e.g., in fast multipole methods.

%
%
\begin{definition}[Cluster tree]
\label{de:cluster}
Let $\mathcal{T}$ be a labeled tree such that the label $\hat t$
of each node $t\in\mathcal{T}$ is a subset of the index set $\Idx$.
We call $\mathcal{T}$ a \emph{cluster tree} for $\Idx$ if
\begin{itemize}
  \item the root $r\in\mathcal{T}$ is labeled $\hat r=\Idx$,
  \item the index sets of siblings are disjoint, i.e.,
    \begin{align*}
      t_1\neq t_2 &\Longrightarrow \hat t_1\cap\hat t_2=\emptyset &
      &\text{ for all } t\in\mathcal{T},\ t_1,t_2\in\sons(t),
         \text{ and}
    \end{align*}
  \item the index sets of a cluster's children are a partition of
    their parent's index set, i.e.,
    \begin{align*}
      \hat t &= \bigcup_{t'\in\sons(t)} \hat t' &
      &\text{ for all } t\in\mathcal{T}
          \text{ with } \sons(t)\neq\emptyset.
    \end{align*}
\end{itemize}
A cluster tree for $\Idx$ is usually denoted by $\ctI$.
Its nodes are called \emph{clusters}.
We denote the \emph{set of leaves} of $\ctI$ by
$\lfI := \{ t\in\ctI\ :\ \sons(t)=\emptyset \}$.
\end{definition}

A cluster tree $\ctI$ can be split into levels:
we let $\ctIl{0}$ be the set containing only the root of $\ctI$
and define
\begin{align*}
  \ctIl{\ell} &:= \{ t'\in\ctI\ :\ t'\in\sons(t) \text{ for a }
                       t\in\ctIl{\ell-1} \} &
  &\text{ for all } \ell\in\bbbn.
\end{align*}
For each cluster $t\in\ctI$, there is exactly one $\ell\in\bbbn_0$
such that $t\in\ctIl{\ell}$.
We call this the \emph{level number} of $t$ and denote it by $\level(t)=\ell$.
The maximal level
\begin{equation*}
  p_\Idx := \max\{ \level(t)\ :\ t\in\ctI \}
\end{equation*}
is called the \emph{depth} of the cluster tree.

Pairs of clusters $(t,s)$ correspond to subsets $\hat t\times\hat s$
of $\Idx\times\Idx$, i.e., to submatrices of $G\in\bbbc^{\Idx\times\Idx}$.
These pairs inherit the hierarchical structure provided by the
cluster tree.

In order to approximate $G|_{\hat t\times\hat s}$, the directional
interpolation approach uses axis-parallel \emph{bounding boxes}
$B_t,B_s\subseteq\bbbr^3$ such that
\begin{align*}
  \supp(\varphi_i) &\subseteq B_t, &
  \supp(\varphi_j) &\subseteq B_s &
  &\text{ for all } i\in\hat t,\ j\in\hat s,
\end{align*}
and constructs an approximation $\tilde g_{ts}$ of $g|_{B_t\times B_s}$.
Discretizing $\tilde g_{ts}$ then gives rise to an approximation of
the submatrix $G|_{\hat t\times\hat s}$.

For large wave numbers $\kappa$, the function $g|_{B_t\times B_s}$ cannot
be expected to be smooth, so we cannot apply interpolation directly.
This problem can be solved by \emph{directional} interpolation
\cite{BR91,ENYI07,MESCDA12}:
we choose a direction $c\in\bbbr^3$ and split the function $g$ into a
plane wave and a remainder term, i.e., we use
\begin{align*}
  g(x,y)
  &= \exp(\bi \kappa \langle x-y, c \rangle)
     \frac{\exp(\bi \kappa (\|x-y\| - \langle x-y, c \rangle))}
          {4\pi \|x-y\|}\\
  &= \exp(\bi \kappa \langle x-y, c \rangle) g_c(x,y),
\end{align*}
where the remainder is defined by
\begin{equation*}
  g_c(x,y) = \frac{\exp(\bi \kappa (\|x-y\| - \langle x-y, c \rangle))}
                  {4\pi \|x-y\|}.
\end{equation*}
This function is smooth \cite{BOME15} and can therefore be interpolated
by polynomials if the following \emph{admissibility conditions} hold:
\begin{subequations}\label{eq:admissibility}
\begin{align}
  \kappa \left\| \frac{m_t-m_s}{\|m_t-m_s\|} - c \right\|
  &\leq \frac{\eta_1}{\max\{\diam(B_t), \diam(B_s)\}},
    \label{eq:adm_directions}\\
  \max\{\diam(B_t), \diam(B_s)\} &\leq \eta_2 \dist(B_t, B_s),\label{eq:adm_standard}\\
  \kappa \max\{\diam(B_t)^2, \diam(B_s)^2\} &\leq \eta_2 \dist(B_t, B_s),\label{eq:adm_parabolic}
\end{align}
\end{subequations}
where $m_t\in B_t$ and $m_s\in B_s$ denote the midpoints of the boxes
and $\eta_1,\eta_2\in\bbbr_{>0}$ are chosen to strike a balance between
fast convergence (if both parameters are small) and low computational
cost (if both parameters are large).
Condition (\ref{eq:adm_directions}) ensures that $c$
is sufficiently close to the direction from the midpoint $m_s$ to the
midpoint $m_t$, condition (\ref{eq:adm_parabolic}) allows us to extend
this property to directions from any point $y\in B_s$ to any point $x\in B_t$
\cite[Lemma~3.9]{BOME15}, while condition (\ref{eq:adm_standard}) is
required to keep admissible blocks sufficiently far away from the
singularity at $x=y$.

Due to \cite[Corollary~3.14]{BOME15}, the interpolating polynomial
\begin{equation*}
  \tilde g_{c,ts}(x,y)
  = \sum_{\nu,\mu=1}^k \mathcal{L}_{t,\nu}(x)
                     g_c(\xi_{t,\nu}, \xi_{s,\mu})
                     \mathcal{L}_{s,\mu}(y)
\end{equation*}
converges exponentially to $g_c$ in $B_t\times B_s$, and
the error is bounded independently of the wavenumber $\kappa$.
Here $(\xi_{t,\nu})_{\nu=1}^k$ and $(\xi_{s,\mu})_{\mu=1}^k$ are families of
tensor interpolation points in $B_t$ and $B_s$, while
$(\mathcal{L}_{t,\nu})_{\nu=1}^k$ and $(\mathcal{L}_{s,\mu})_{\mu=1}^k$ are
the corresponding families of tensor Lagrange polynomials.

Multiplying by the plane wave, we obtain
\begin{align*}
  g(x,y)
  &= \exp(\bi \kappa \langle x-y, c \rangle) g_c(x,y)\\
  &\approx \exp(\bi \kappa \langle x-y, c \rangle)
     \sum_{\nu,\mu=1}^k \mathcal{L}_{t,\nu}(x)
                      g_c(\xi_{t,\nu}, \xi_{s,\mu})
                      \mathcal{L}_{s,\mu}(x)\\
  &= \sum_{\nu,\mu=1}^k
     \exp(\bi \kappa \langle x, c \rangle) \mathcal{L}_{t,\nu}(x)
     g_c(\xi_{t,\nu}, \xi_{s,\mu})
     \overline{
     \exp(\bi \kappa \langle y, c \rangle) \mathcal{L}_{s,\mu}(y)}\\
  &= \sum_{\nu,\mu=1}^k
     \mathcal{L}_{tc,\nu}(x)
     g_c(\xi_{t,\nu}, \xi_{s,\mu})
     \overline{\mathcal{L}_{sc,\mu}(y)} =: \tilde g_b(x,y)
   \qquad\text{ for all } x\in B_t,\ y\in B_s
\end{align*}
with the modified Lagrange polynomials
\begin{align*}
  \mathcal{L}_{tc,\nu}(x)
  &= \exp(\bi \kappa \langle x, c \rangle) \mathcal{L}_{t,\nu}(x), &
  \mathcal{L}_{sc,\mu}(y)
  &= \exp(\bi \kappa \langle y, c \rangle) \mathcal{L}_{s,\mu}(y).
\end{align*}
Replacing $g$ by $\tilde g_b$ in (\ref{eq:matrix}) yields
\begin{align*}
  g_{ij} &\approx \int_\Omega \varphi_i(x) \int_\Omega
              \tilde g_b(x,y) \varphi_j(y) \,dy\,dx\\
  &= \sum_{\nu=1}^k \sum_{\mu=1}^k
         \underbrace{g_c(\xi_{t,\nu}, \xi_{s,\mu})}_{=:s_{b,\nu\mu}}
         \underbrace{\int_\Omega \varphi_i(x)
                                \mathcal{L}_{tc,\nu}(x) \,dx}_{=:v_{tc,i\nu}}
         \underbrace{\int_\Omega \varphi_j(y)
                                \overline{\mathcal{L}_{sc,\mu}(y)} \,dy}_{
                     \overline{v_{sc,j\mu}}}\\
  &= \sum_{\nu=1}^k \sum_{\mu=1}^k s_{b,\nu\mu}
         v_{tc,i\nu} \overline{v_{sc,j\mu}}
   = (V_{tc} S_b V_{sc}^*)_{ij}
   \qquad\text{ for all } i\in\hat t,\ j\in\hat s
\end{align*}
with matrices $V_{tc}\in\bbbc^{\hat t\times k}$,
$V_{sc}\in\bbbc^{\hat s\times k}$, and $S_b\in\bbbc^{k\times k}$.
This is a factorized low-rank approximation
\begin{equation}\label{eq:matrix_apx}
  G|_{\hat t\times\hat s} \approx V_{tc} S_b V_{sc}^*
\end{equation}
of the submatrix $G|_{\hat t\times\hat s}$.

Since the matrix $G$ itself does not satisfy the conditions
(\ref{eq:admissibility}), we have to split it into submatrices, and
experiments show that the number of submatrices grows rapidly as the
problem size increases.
Storing the matrices $(V_{tc})_{t\in\ctI}$ for all clusters $t\in\ctI$
would lead to quadratic complexity and is therefore unattractive.
Fortunately, we can take advantage of the hierarchical structure of
the cluster tree if we organize the directions $c$ accordingly.

%
%
\begin{definition}[Directions]
\label{de:directions}
Let $(\mathcal{D}_\ell)_{\ell=0}^{\pI}$ be a family of finite subsets
of $\bbbr^3$.
It is called a \emph{family of directions} if
\begin{align*}
  \|c\|=1 &\vee c=0 &
  &\text{ for all } c\in\mathcal{D}_\ell,\ \ell\in[0:\pI].
\end{align*}
Here the special case $c=0$ is included to allow us to treat
the low-frequency case that does not require us to split off a
plane wave.
\end{definition}

Given a family of directions, we fix a family
$(\sd{\ell})_{\ell=0}^{\pI-1}$ of mappings
$\sd{\ell}:\mathcal{D}_\ell\to\mathcal{D}_{\ell+1}$ such that
\begin{align*}
  \|c - \sd{\ell}(c)\| &\leq \|c-\tilde c\| &
  &\text{ for all } c\in\mathcal{D}_\ell,\ \tilde c\in\mathcal{D}_{\ell+1},
                     \ \ell\in[0:\pI-1].
\end{align*}
Given a cluster tree $\ctI$, we write
\begin{align*}
  \mathcal{D}_t &:= \mathcal{D}_{\level(t)}, &
  \sd{t}(c) &:= \sd{\level(t)}(c) &
  &\text{ for all } t\in\ctI,\ c\in\mathcal{D}_{\level(t)}.
\end{align*}

In order to satisfy the admissibility condition
(\ref{eq:adm_directions}), we use sets of directions that are
sufficiently large to approximate \emph{any} direction, i.e., we require
\begin{align}\label{eq:all_directions}
  \min\{ \|y-c\|\ :\ c\in\mathcal{D}_t \}
  &\leq \frac{\eta_1}{\kappa \diam(B_t)} &
  &\text{ for all } t\in\ctI,\ y\in\bbbr^3 \text{ with } \|y\|=1.
\end{align}
Since the size of clusters decreases as the level grows, this
requirement means that large clusters will require more directions
than small clusters.

%
%
\begin{remark}[Construction of directions]
\label{re:construction_of_directions}
In our implementation, we construct sets of directions
satisfying (\ref{eq:all_directions}) as follows:
for a level $\ell\in[0:\pI]$, we compute the maximal diameter
$d_\ell$ of all bounding boxes $B_t$ associated
with clusters $t\in\ctIl{\ell}$ on this level.
If $\kappa d_\ell \leq \eta_1$ holds, we let $\mathcal{D}_\ell = \{ 0 \}$,
i.e., we use no directional approximation in the low-frequency case.

Otherwise, i.e., if $\kappa d_\ell > \eta_1$, we let
$m := \lceil \nicefrac{\sqrt{2} \kappa d_\ell}{\eta_1} \rceil$
and split each side of the unit cube $[0,1]^3$ into $m^2$ squares of width
$\nicefrac{2}{m}$ and diameter $\nicefrac{2 \sqrt{2}}{m}$.
Since the cube has six sides, we have a total of $6 m^2\in\mathcal{O}(\kappa^2
d_\ell^2)$ such squares.
We denote the centers of the squares by $\tilde c_\iota$, and their
projections to the unit sphere by
$c_\iota := \nicefrac{\tilde c_\iota}{\|\tilde c_\iota\|_2}$
for $\iota\in[1:6 m^2]$.
We let $\mathcal{D}_\ell := \{ c_\iota\ :\ \iota\in[1:6m^2] \}$.
For every unit vector $y\in\bbbr^3$, there is a point $\tilde y$ on the
surface of the unit cube with $y = \nicefrac{\tilde y}{\|\tilde y\|_2}$.
Since the surface grid is sufficiently fine, we can find
$\iota\in[1:6m^2]$
with $\|\tilde y-\tilde c_\iota\|_2 \leq \nicefrac{\sqrt{2}}{m}$,
and the projection ensures
\begin{align*}
  \|y-c_\iota\|_2
  &\leq \|\tilde y-\tilde c_\iota\|_2
   \leq \frac{\sqrt{2}}{m}
   \leq \frac{\eta_1}{\kappa d_\ell}
   \leq \frac{\eta_1}{\kappa \diam(B_t)} &
  &\text{ for all } t\in\ctIl{\ell},
\end{align*}
i.e., (\ref{eq:all_directions}) holds.
By this construction, the set $\mathcal{D}_\ell$ is sufficiently
large to contain approximations for \emph{any} unit vector, but still small
enough to satisfy the assumption (\ref{eq:directions_cluster}) required for
our complexity analyis.
\end{remark}

Due to (\ref{eq:all_directions}), we only have to satisfy
the conditions (\ref{eq:adm_standard}) and (\ref{eq:adm_parabolic}) and
can then find a direction $c_{ts}\in\mathcal{D}_t=\mathcal{D}_s$ that
satisfies the first condition (\ref{eq:adm_directions}):
for $t,s\in\ctIl{\ell}$, we let $c_{ts}\in\mathcal{D}_\ell$
be a best approximation of the direction from the midpoint $m_s$ of
the source box $B_s$ to midpoint $m_t$ of the target box $B_t$, i.e.,
\begin{align*}
  \left\| \frac{m_t-m_s}{\|m_t-m_s\|_2} - c_{ts} \right\|_2
  &\leq \left\| \frac{m_t-m_s}{\|m_t-m_s\|_2} - \tilde c \right\|_2 &
  &\text{ for all } \tilde c\in\mathcal{D}_\ell.
\end{align*}
If $m_t=m_s$, the admissibility condition (\ref{eq:adm_standard})
is violated and we can choose any $c_{ts}\in\mathcal{D}_\ell$.
This leaves us with the task of splitting the matrix $G$ into
submatrices $G|_{\hat t\times\hat s}$ such that $B_t$ and $B_s$ satisfy
the admissibility conditions (\ref{eq:adm_standard}) and
(\ref{eq:adm_parabolic}).
A decomposition with the minimal necessary number of submatrices
can be constructed by a recursive procedure that again gives rise
to a tree structure and inductively ensures $\level(t)=\level(s)$,
so that the directions $c_{ts}$ are well-defined.

%
%
\begin{definition}[Block tree]
\label{de:block}
Let $\ctI$ be a cluster tree for the index set $\Idx$ with root
$r_\Idx$,
let $(\mathcal{D}_\ell)_{\ell=0}^{\pI}$ be a family of
directions.

A tree $\mathcal{T}$ is called a \emph{block tree} for $\ctI$ if
\begin{itemize}
  \item for each $b\in\mathcal{T}$ there are $t,s\in\ctI$ such
        that $b=(t,s,c_{ts})$,
  \item the root $r\in\mathcal{T}$ satisfies $r=(r_\Idx,r_\Idx,c_{r_\Idx r_\Idx})$,
  \item for each $b=(t,s,c_{ts})\in\mathcal{T}$ we have
    \begin{equation}\label{eq:block_sons}
      \sons(b)\neq\emptyset \Longrightarrow
      \sons(b) = \{ (t',s',c_{t's'})\ :\ t'\in\sons(t),
                      \ s'\in\sons(s) \}.
    \end{equation}
\end{itemize}
A block tree for $\ctI$ is usually denoted by $\ctII$.
Its nodes are called \emph{blocks}.
We denote the set of leaves of $\ctII$ by
$\lfII := \{ b\in\ctII\ :\ \sons(b)=\emptyset \}$.
\end{definition}

The leaves of a block tree define a disjoint partition of
the index set $\Idx\times\Idx$, i.e., a decomposition of the
matrix $G\in\bbbc^{\Idx\times\Idx}$ into submatrices
$G|_{\hat t\times\hat s}$ with $(t,s,c)\in\lfII$.

We can construct a block tree with the minimal number of
blocks by a simple recursion:
starting with the root, we check whether a block is admissible.
If it is, we make it an admissible leaf and represent the
corresponding submatrix in the factorized form (\ref{eq:matrix_apx}).
Otherwise, we consider its children given by (\ref{eq:block_sons}).
If there are no children, i.e., if $\sons(t)$ or $\sons(s)$ are
empty, we have found an inadmissible leaf and store the submatrix
directly, i.e., as a two-dimensional array.

While the approximation (\ref{eq:matrix_apx}) reduces the amount
of storage required for one block to $k^2$, we still have to store
the matrices $(V_{tc})_{t\in\ctI, c\in\mathcal{D}_t}$, and in the
high-frequency case the storage requirements for these matrices
grow at least quadratically with the problem
size: if we have $\kappa^2 \sim n$, doubling the matrix
dimension means multiplying $\kappa$ by a factor of $\sqrt{2}$.
Constructing directions as in Remark~\ref{re:construction_of_directions},
the splitting parameter $m$ also grows by a factor of approximately
$\sqrt{2}$, and the number of directions is therefore doubled, as well.
On a given level $\ell\in[0:\pI]$, storing $V_{tc}$ for all $t\in\ctI$
with a \emph{fixed} direction $c$ requires $\mathcal{O}(nk)$ coefficients,
and since we have $\mathcal{O}(n)$ directions, we even end up with
$\mathcal{O}(n^2k)$ coefficients \emph{per level}.

In order to solve this problem, we take advantage of
the requirement (\ref{eq:all_directions}):
given a cluster $t\in\ctI$, a direction $c\in\mathcal{D}_t$, and one of
its children $t'\in\sons(t)$, we can find a direction
$c':=\sd{\ell}(c)\in\mathcal{D}_{t'}$ that approximates $c$
reasonably well.
This property allows us to reduce the storage requirements for
the matrices $(V_{tc})_{t\in\ctI,c\in\mathcal{D}_t}$ as follows:
since $\|c-c'\|_2$ is small, the function
\begin{equation*}
  x \mapsto \exp(-\bi \kappa \langle x, c' \rangle)
            \mathcal{L}_{tc,\nu}(x)
          = \exp(\bi \kappa \langle x, c-c' \rangle)
            \mathcal{L}_{t,\nu}(x)
\end{equation*}
is smooth and can therefore be interpolated in $B_{t'}$.
We find
\begin{align*}
  \mathcal{L}_{tc,\nu}(x)
  &= \exp(\bi \kappa \langle x, c \rangle) \mathcal{L}_{t,\nu}(x)
   = \exp(\bi \kappa \langle x, c' \rangle)
     \exp(\bi \kappa \langle x, c-c' \rangle) \mathcal{L}_{t,\nu}(x)\\
  &\approx \exp(\bi \kappa \langle x, c' \rangle)
     \sum_{\nu'=1}^k
        \underbrace{\exp(\bi \kappa \langle \xi_{t',\nu'}, c-c' \rangle)
                    \mathcal{L}_{t,\nu}(\xi_{t',\nu'})}_{=:e_{t'c,\nu'\nu}}
                    \mathcal{L}_{t',\nu'}(x)\\
  &= \sum_{\nu'=1}^k e_{t'c,\nu'\nu} \mathcal{L}_{t'c',\nu'}(x).
\end{align*}
This approach immediately yields
\begin{align*}
  v_{tc,i\nu}
  &= \int_\Omega \varphi_i(x) \mathcal{L}_{tc,\nu}(x) \,dx
   \approx \sum_{\nu'=1}^k e_{t'c,\nu'\nu}
             \int_\Omega \varphi_i(x) \mathcal{L}_{t'c',\nu'}(x) \,dx
   = (V_{t'c'} E_{t'c})_{i\nu}
\end{align*}
for all $i\in\hat t'$ and $\nu\in[1:k]$, which is equivalent
to
\begin{equation}\label{eq:nested_apx}
  V_{tc}|_{\hat t'\times k} \approx V_{t'c'} E_{t'c}.
\end{equation}
Instead of storing $V_{tc}$, we can just store small $k\times k$ matrices
$E_{t'c}$ for all $t'\in\sons(t)$, thus reducing the storage requirements
from $(\#\hat t) k$ to $\mathcal{O}(k^2)$.
This approach implies using the right-hand side of
(\ref{eq:nested_apx}) to \emph{define} the left-hand side.

%
%
\begin{definition}[Directional cluster basis]
Let $k\in\bbbn$, and let $V=(V_{tc})_{t\in\ctI,c\in\mathcal{D}_t}$ be a family
of matrices.
We call it a \emph{directional cluster basis} if
\begin{itemize}
  \item $V_{tc}\in\bbbc^{\hat t\times k}$ for all $t\in\ctI$
     and $c\in\mathcal{D}_t$, and
  \item there is a family
     $E=(E_{t'c})_{t\in\ctI,t'\in\sons(t),c\in\mathcal{D}_t}$
     such that
     \begin{align}\label{eq:nested}
       V_{tc}|_{\hat t'\times k} &= V_{t'c'} E_{t'c} &
       &\text{ for all } t\in\ctI,\ t'\in\sons(t),\ c\in\mathcal{D}_t,
                        \ c'=\sd{t}(c).
     \end{align}
\end{itemize}
The elements of the family $E$ are called \emph{transfer matrices} for
the directional cluster basis $V$, and $k$ is called its \emph{rank}.
\end{definition}

%
%
\begin{remark}[Notation]
The notation ``$E_{t'c}$'' for the transfer matrices (instead of something
like ``$E_{t'tc'c}$'' listing all parameters) for the matrices is justified
since the parent $t\in\ctI$ is uniquely determined by $t'\in\ctI$ due to
the tree structure, and the direction $c'=\sd{t}(c)$ is uniquely determined
by $c\in\mathcal{D}_t$ due to our Definition~\ref{de:directions}.
\end{remark}

We can now define the class of matrices that is the subject of this
article:
since the leaves $\lfII$ of the block tree correspond to a partition
of the matrix $G$, we have to represent each of the submatrices
$G|_{\hat t\times\hat s}$ for $b=(t,s,c)\in\lfII$.
Those blocks that satisfy the admissibility conditions
(\ref{eq:admissibility}) can be approximated in the form
(\ref{eq:matrix_apx}).
These matrices are called \emph{admissible} and collected in a
subset
\begin{align*}
  \lfaII &:= \{ b\in\lfII\ :\ b \text{ is admissible} \}.\\
\intertext{The remaining blocks are called \emph{inadmissible} and
collected in the set}
  \lfiII &:= \lfII \setminus \lfaII.
\end{align*}
These matrices are stored as simple two-dimensional arrays without
any compression.

%
%
\begin{definition}[Directional $\mathcal{H}^2$-matrix]
\label{de:dh2matrix}
Let $V$ and $W$ be directional cluster bases for $\ctI$.
Let $G\in\bbbc^{\Idx\times\Idx}$ be a matrix.
We call it a \emph{directional $\mathcal{H}^2$-matrix} (or just a
\emph{$\mathcal{DH}^2$-matrix}) if there are families
$S=(S_b)_{b\in\lfaII}$ such that
\begin{align}\label{eq:vsw}
  G|_{\hat t\times\hat s} &= V_{tc} S_b W_{sc}^* &
  &\text{ for all } b=(t,s,c)\in\lfaII.
\end{align}
The elements of the family $S$ are called \emph{coupling matrices}.
$V$ is called the \emph{row cluster basis} and $W$ is called the
\emph{column cluster basis}.

A \emph{$\mathcal{DH}^2$-matrix representation} of a
$\mathcal{DH}^2$-matrix $G$ consists of $V$, $W$, $S$ and the
family $(G|_{\hat b})_{b\in\lfiII}$ of \emph{nearfield matrices} corresponding
to the inadmissible leaves of $\ctII$.
\end{definition}

Under typical assumptions, including that $k$ is fixed
independently of $\kappa$, since the interpolation error does not
depend on $\kappa$, it is possible to prove that a
$\mathcal{DH}^2$-matrix requires only
$\mathcal{O}(n k + \kappa^2 k^2 \log(n))$ units of storage
\cite[Section~5]{BO15}.

\section{Recompression}


\subsection{Compression of general matrices}
\label{se:compression_general}

Before we address the recompression of a $\mathcal{DH}^2$-matrix,
we briefly recall the compression algorithm for general matrices
described in \cite{BO15}.

Let $G\in\bbbc^{\Idx\times\Idx}$.
We want to approximate the matrix by an \emph{orthogonal projection},
since this guarantees optimal stability and best-approximation properties
with respect to certain norms.

We call a matrix $X\in\bbbc^{\Idx\times\Kdx}$ \emph{isometric} if
$X^* X = I$ holds.
If $X$ is isometric, $X X^*$ is the orthogonal projection into the
range of $X$, i.e., it maps a vector $y\in\bbbc^\Idx$ onto its best
approximation $\widetilde{y} := X X^* y$ in this space, and the stability
estimate $\|\widetilde{y}\|_2 \leq \|y\|_2$ holds.

We call the cluster bases $(V_{tc})_{t\in\ctI,c\in\mathcal{D}_t}$ and
$(W_{tc})_{t\in\ctI,c\in\mathcal{D}_t}$ \emph{orthogonal} if all matrices
are isometric, i.e., if
\begin{align*}
  V_{tc}^* V_{tc} &= I, &
  W_{tc}^* W_{tc} &= I &
  &\text{ holds for all } t\in\ctI,\ c\in\mathcal{D}_t.
\end{align*}
In this case, the optimal coupling matrices with respect to the
Frobenius norm (and almost optimal with respect to the spectral norm)
can be computed by orthogonal projection using
\begin{equation*}
  G|_{\hat t\times\hat s}
  \approx V_{tc} V_{tc}^* G|_{\hat t\times\hat s} W_{sc} W_{sc}^*
  = V_{tc} S_b W_{tc}^*
  \qquad\text{ with }\qquad
  S_b := V_{tc}^* G|_{\hat t\times\hat s} W_{sc}.
\end{equation*}
Due to
\begin{align*}
  \|G|_{\hat t\times\hat s}
    - V_{tc} S_b W_{sc}^*\|_F^2
  &= \|G|_{\hat t\times\hat s} - V_{tc} V_{tc}^* G|_{\hat t\times\hat s}\|_F^2
   + \|V_{tc} V_{tc}^* (G|_{\hat t\times\hat s}
                      - G|_{\hat t\times\hat s} W_{sc} W_{sc}^*) \|_F^2\\
  &\leq \|G|_{\hat t\times\hat s}
           - V_{tc} V_{tc}^* G|_{\hat t\times\hat s}\|_F^2
      + \|G|_{\hat t\times\hat s}^*
           - W_{sc} W_{sc}^* G|_{\hat t\times\hat s}^*\|_F^2,
\end{align*}
we can focus on the construction of a good row cluster basis, since
a good column cluster basis can be obtained by applying the same
procedure to the adjoint matrix.

By Definition~\ref{de:dh2matrix}, the matrix $V_{tc}$ has to be able
to approximate the range of all matrices $G|_{\hat t\times\hat s}$
with $(t,s,c)\in\lfaII$.
We collect the corresponding column clusters in the set
\begin{equation*}
  \brow(t,c) := \{ s\in\ctI\ :\ (t,s,c)\in\lfaII \}.
\end{equation*}
We also have to take the nested structure of the cluster basis
into account.
Let $t\in\ctI$ with $\sons(t)\neq\emptyset$.
For the sake of simplicity, we focus on the case $\#\sons(t)=2$ and
$\sons(t)=\{t_1,t_2\}$.
Assume that isometric matrices $V_{t_1 c_1}$ and $V_{t_2 c_2}$ with
$c_1=\sd{t_1}(c)$ and $c_2=\sd{t_2}(c)$ have already been computed.
Due to (\ref{eq:nested}), we have
\begin{equation}\label{eq:V_Vhat}
  V_{tc}
  = \begin{pmatrix}
      V_{t_1 c_1} & \\
      & V_{t_2 c_2}
    \end{pmatrix}
    \widehat{V}_{tc}
  \qquad\text{ with }\qquad
  \widehat{V}_{tc}
  := \begin{pmatrix}
       E_{t_1 c}\\ E_{t_2 c}
     \end{pmatrix}.
\end{equation}
The error of the orthogonal projection takes the form
\begin{equation*}
  \|G|_{\hat t\times\hat s} - V_{tc} V_{tc}^* G|_{\hat t\times\hat s}\|_F^2
  = \left\|G|_{\hat t\times\hat s}
            - \begin{pmatrix}
                V_{t_1 c_1} & \\
                & V_{t_2 c_2}
              \end{pmatrix} \widehat{V}_{tc}
              \widehat{V}_{tc}^*
              \begin{pmatrix}
                V_{t_1 c_1}^* & \\
                & V_{t_2 c_2}^*
              \end{pmatrix} G|_{\hat t\times\hat s}\right\|_F^2.
\end{equation*}
Since $V_{t_1 c_1}$ and $V_{t_2 c_2}$ are assumed to be isometric,
Pythagoras' identity yields
\begin{align}
  \|G|_{\hat t\times\hat s} - V_{tc} V_{tc}^* G|_{\hat t\times\hat s}\|_F^2
  &= \left\|G|_{\hat t\times\hat s}
            - \begin{pmatrix}
                V_{t_1 c_1} & \\
                & V_{t_2 c_2}
              \end{pmatrix}
              \begin{pmatrix}
                V_{t_1 c_1}^* & \\
                & V_{t_2 c_2}^*
              \end{pmatrix} G|_{\hat t\times\hat s}\right\|_F^2\notag\\
  &\quad + \left\| \begin{pmatrix}
               V_{t_1 c_1} & \\
               & V_{t_2 c_2}
             \end{pmatrix}
             ( I - \widehat{V}_{tc} \widehat{V}_{tc}^* )
             \begin{pmatrix}
               V_{t_1 c_1}^* & \\
               & V_{t_2 c_2}^*
             \end{pmatrix} G|_{\hat t\times\hat s}\right\|_F^2\notag\\
  &= \|G|_{\hat t_1\times s}
       - V_{t_1 c_1} V_{t_1 c_1}^* G|_{\hat t_1\times s}\|_F^2\notag\\
  &\quad + \|G|_{\hat t_2\times s}
       - V_{t_2 c_2} V_{t_2 c_2}^* G|_{\hat t_2\times s}\|_F^2\notag\\
  &\quad + \left\|(I - \widehat{V}_{tc} \widehat{V}_{tc}^*)
                  \begin{pmatrix}
                    V_{t_1 c_1}^* G|_{\hat t_1\times s}\\
                    V_{t_2 c_2}^* G|_{\hat t_2\times s}
                  \end{pmatrix} \right\|_F^2.\label{eq:error_sons}
\end{align}
We can see that the projection error for the cluster $t$ depends
on the projection errors for its children $t_1$ and $t_2$.
Using a straightforward induction, we find that all descendants
of $t$ contribute to the error.

This means that our algorithm has to take all ancestors of a
cluster $t$ into account when it constructs $V_{tc}$.
We collect these ancestors and the corresponding directions in
the sets
\begin{align}\label{eq:ancestors}
  \anc(t,c)
  &:= \begin{cases}
        \{ (t,c) \} &\text{ if } t \text{ is the root of } \ctI,\\
        \{ (t,c) \}
        \cup \bigcup_{c^+\in\sd{t^+}^{-1}(\{c\})}
             \anc(t^+,c^+)
        &\text{ if } t \text{ is a child of } t^+\in\ctI
      \end{cases}
\end{align}
for all $t\in\ctI$ and $c\in\mathcal{D}_t$.
The mapping $\sd{t^+}$ is not invertible, since the parent cluster
frequently is associated with more directions than its child.
This is the reason there can be multiple $c^+\in\mathcal{D}_{t^+}$
with $\sd{t^+}(c^+)=c$ in (\ref{eq:ancestors}).
We have to find $V_{tc}$ such that
\begin{align}\label{eq:approx_submatrices}
  G|_{\hat t\times\hat s}
  &\approx V_{tc} V_{tc}^* G|_{\hat t\times\hat s} &
  &\text{ for all } s\in\brow(\tilde t,\tilde c),
                    \ (\tilde t,\tilde c)\in\anc(t,c).
\end{align}
Due to Definition~\ref{de:block}, the index sets of the clusters in
\begin{equation*}
  \brow^+(t,c) := \bigcup_{(\tilde t,\tilde c)\in\anc(t,c)}
                    \brow(\tilde t,\tilde c)
\end{equation*}
are disjoint, and we can introduce
\begin{align*}
  \mathcal{R}_{tc} &:= \bigcup_{s\in\brow^+(t,c)} \hat s, &
  G_{tc} &:= G|_{\hat t\times\mathcal{R}_{tc}} &
  &\text{ for all } t\in\ctI,\ c\in\mathcal{D}_t
\end{align*}
in order to rewrite (\ref{eq:approx_submatrices}) in the form
\begin{equation*}
  G_{tc} \approx V_{tc} V_{tc}^* G_{tc}.
\end{equation*}
If $t$ is a leaf cluster, we can directly find the optimal
approximation by computing the SVD of $G_{tc}$
and using the first $k$ left singular vectors as the columns of
the matrix $V_{tc}$:
the SVD yields an orthonormal basis $(v_i)_{i=1}^\tau$ of left singular
vectors, an orthonormal basis $(u_i)_{i=1}^\tau$ of right singular vectors,
and ordered singular values $\sigma_{1} \geq \dots \geq \sigma_\tau \geq 0$
with
\begin{align*}
  G_{tc} &= \sum_{i=1}^\tau v_i \sigma_i u_i^*,
\end{align*}
where $\tau = \#\hat t$.
Using this notation, it is an easy task to find the lowest rank
$k\in[0:\tau]$ such that the approximation
\begin{align*}
  \widetilde{G}_{tc} &:= \sum_{i=1}^k v_i \sigma_i u_i^*
  = V_{tc} V_{tc}^* G_{tc}, &
  V_{tc} &:= \begin{pmatrix} v_1 & \ldots & v_k \end{pmatrix},
\end{align*}
still ensures the desired error bound \cite[Lemma~5.19]{BO10}.
For the sake of simplicity, we consider only the Frobenius norm case,
the spectral norm and relative errors bounds are available with slight
adaptations in the choice of $k$ \cite[Theorem~2.5.3]{GOVL96}.

If $t$ is not a leaf cluster, (\ref{eq:error_sons}) indicates that
we have to look for $\widehat{V}_{tc}$ such that
\begin{equation*}
  \widehat{G}_{tc} \approx \widehat{V}_{tc} \widehat{V}_{tc}^* \widehat{G}_{tc}
\end{equation*}
with the matrix
\begin{equation}\label{eq:Ghat}
  \widehat{G}_{tc}
  := \begin{pmatrix}
       V_{t_1 c_1}^* & \\
       & V_{t_2 c_2}^*
     \end{pmatrix} G_{tc}
\end{equation}
containing the coefficients for the approximation of $G_{tc}$ in the
children's bases.
This task can again be solved by computing the SVD of $\widehat{G}_{tc}$,
which takes only $\mathcal{O}(k^2 \#\mathcal{R}_{tc})$ operations, and
the transfer matrices $E_{t_1 c}$ and $E_{t_2 c}$ can be
obtained from $\widehat{V}_{tc}$ by definition (\ref{eq:V_Vhat}).


\subsection{\texorpdfstring{$\mathcal{DH}^2$-recompression}
                           {DH2-recompression}}

The algorithm presented in the previous section has quadratic
complexity if we assume that the numerical ranks $k$ are
uniformly bounded, cf. \cite[Theorem~17]{BO15}, since it starts
with a dense matrix $G$ represented explicitly by $n^2$ coefficients.
This means that the algorithm is only of theoretical interest, i.e.,
for investigating whether a given matrix can be approximated at all,
but not attractive for real applications with large numbers
of degrees of freedom.

In the case of the Helmholtz equation, it has already been proven
\cite{BOME15} that directional interpolation provides us with an
$\mathcal{DH}^2$-matrix approximation, although the rank of this
approximation may be larger than necessary.
Our task is therefore only to \emph{recompress} an already
compressed $\mathcal{DH}^2$-matrix, we do not have to start from
scratch.
If we can arrange the algorithm in a way that avoids creating
the entire original approximation, we can obtain nearly optimal
storage requirements without the need of excessive storage for
intermediate results.

Our first step is to take advantage of the $\mathcal{DH}^2$-matrix
structure to reduce the complexity of our algorithm.
We assume that the original matrix is described by cluster bases
$(V_{tc})_{t\in\ctI,c\in\mathcal{D}_t}$,
$(W_{tc})_{t\in\ctI,c\in\mathcal{D}_t}$ and coupling matrices
$(S_{b})_{b\in\lfaII}$ such that
\begin{align*}
  G|_{\hat t\times\hat s}
  &= V_{tc} S_{b} W_{sc}^* &
  &\text{ for all } b=(t,s,c)\in\lfaII.
\end{align*}
We denote the transfer matrices of the cluster
basis $(V_{tc})_{t\in\ctI,c\in\mathcal{D}_t}$ by $E_{t'c}\in\bbbc^{k\times k}$ for
$t\in\ctI$, $t'\in\sons(t)$, $c\in\mathcal{D}_t$, and $c'=\sd{t}(c)$.

Our goal is to find improved row and column cluster bases
for the matrix $G$ by the algorithm outlined in
Section~\ref{se:compression_general}, but to take advantage of
the $\mathcal{DH}^2$-matrix structure of $G$ to reduce the
computational work.

We have already seen that we only have to describe an algorithm
for computing row bases, since applying this algorithm to the
adjoint matrix $G^*$ will yield a column basis, as well.
We call the improved row basis $(Q_{tc})_{t\in\ctI,c\in\mathcal{D}_t}$,
the adaptively-chosen rank of $Q_{tc}$ is called $k_{tc}$, and the
transfer matrices are called $(F_{t'c})_{t\in\ctI,t'\in\sons(t),c\in\mathcal{D}_t}$.

In particular, the matrices $V_{tc}$ and $W_{tc}$ are no longer
isometric, and ensuring reliable error control requires the use
of suitable weight matrices.
Since we assume that $V_{tc}$ and $W_{tc}$ result from directional
interpolation of constant order, we know that all of these matrices
have a fixed number $k$ of columns.

Our approach to speeding up the algorithm of
Section~\ref{se:compression_general} is to obtain a factorized
low-rank representation of the matrices $G_{tc}$ required by the
compression algorithm that allows us to efficiently compute an
improved basis.
In particular, we will prove that there are $k\times k$ matrices
$\widehat{Z}_{tc}$ for all $t\in\ctI$, $c\in\mathcal{D}_t$ such that
\begin{equation}\label{eq:recomp_weight_full}
  G_{tc} = V_{tc} \widehat{Z}_{tc}^* P_{tc}^*
\end{equation}
holds with an isometric matrix $P_{tc}\in\bbbc^{\mathcal{R}_{tc}\times k}$.
Since $P_{tc}$ is isometric, it does not influence the left
singular vectors or the non-zero singular values, so we can replace
$G_{tc}$ by the skinny matrix $V_{tc} \widehat{Z}_{tc}^*$ in the
compression algorithm of Section~\ref{se:compression_general}
and construct $Q_{tc}$ from the left singular vectors of this smaller
matrix without changing the result.

Let $t\in\ctI$, $c\in\mathcal{D}_t$.
For the moment, we assume that $t$ is not the root of the cluster
tree, i.e., that it has a parent $t^+\in\ctI$ with $t\in\sons(t^+)$.

We assume that the matrices $\widehat{Z}_{t^+c^+}$ have already been computed
for all directions $c^+\in\mathcal{D}_{t^+}$ with $\sd{t^+}(c^+) = c$, i.e.,
for all $c^+\in\sd{t^+}^{-1}(\{c\})$.
Let $\gamma := \#\sd{t^+}^{-1}(\{c\})$ denote the number of directions
in $\mathcal{D}_{t^+}$ that get mapped to $c$, and enumerate these
directions as $\sd{t^+}^{-1}(\{c\}) = \{c_1^+,\ldots,c_\gamma^+\}$.
Due to definition (\ref{eq:ancestors}), we have
\begin{equation*}
  \anc(t,c) = \{(t,c)\} \cup \bigcup_{\iota=1}^\gamma \anc(t^+,c_\iota^+).
\end{equation*}
We let $\sigma:=\#\brow(t,c)$ and $\brow(t,c)=\{s_1,\ldots,s_\sigma\}$.

Let now $s\in\brow^+(t,c)$.
By definition, we can either have $s\in\brow(t,c)$, or there is
a $\iota\in[1:\gamma]$ such that $s\in\brow^+(t^+,c_\iota^+)$.
In the first case, we have
\begin{equation*}
  G|_{\hat t\times\hat s}
  = V_{tc} S_b W_{sc}^*,
\end{equation*}
and we collect all of these matrices in an auxiliary matrix
\begin{align*}
  H_{tc}
  &:= \begin{pmatrix}
        V_{tc} S_{t s_1 c}
           W_{s_1 c}^* &
        \cdots &
        V_{tc} S_{t s_\sigma c}
           W_{s_\sigma c}^*
      \end{pmatrix}\\
  &= V_{tc}
      \underbrace{\begin{pmatrix}
         S_{t s_1 c} W_{s_1 c}^* &
         \cdots &
         S_{t s_\sigma c} W_{s_\sigma c}^*
       \end{pmatrix}}_{=:Y_{tc}}.
\end{align*}
The matrix $Y_{tc}$ has too many columns for a practical algorithm,
so we use the orthogonalization algorithm \cite[Algorithm~16]{BO10},
with a straightforward generalization, to find $k\times k$ matrices
$R_{W,s_i c}$ and isometric matrices $P_{W,s_i c}$ with
\begin{align}\label{eq:basis_weights_def}
  W_{s_i c} &= P_{W,s_i c} R_{W,s_i c} &
  &\text{ for all } i\in[1:\sigma]
\end{align}
and obtain
\begin{equation*}
  Y_{tc}
  = \underbrace{\begin{pmatrix}
      S_{t s_1 c} R_{W,s_1 c}^* &
      \cdots &
      S_{t s_\sigma c} R_{W,s_\sigma c}^*
    \end{pmatrix}}_{=:\widehat{Y}_{tc}}
    \underbrace{\begin{pmatrix}
      P_{W,s_1 c} & & \\
      & \ddots & \\
      & & P_{W,s_\sigma c}
    \end{pmatrix}^*}_{=:P_{Y,tc}^*}
  = \widehat{Y}_{tc} P_{Y,tc}^*.
\end{equation*}
The matrix $\widehat{Y}_{tc}$ is now sufficiently small, and the
isometric matrix $P_{Y,tc}$ can later be subsumed in $P_{tc}$.

In the second case, i.e., if $s\in\brow^+(t^+,c_\iota^+)$, we have
\begin{equation*}
  G|_{\hat t\times\hat s}
  = G_{t^+ c_\iota^+}|_{\hat t\times\hat s}.
\end{equation*}
Combining both cases yields
\begin{equation*}
  G_{tc}
  = \begin{pmatrix}
      H_{tc} &
      G_{t^+ c_1^+}|_{\hat t\times\mathcal{R}_{t^+ c_1^+}} &
      \cdots &
      G_{t^+ c_\gamma^+}|_{\hat t\times\mathcal{R}_{t^+ c_\gamma^+}}
    \end{pmatrix}.
\end{equation*}
Due to our assumption, we have low-rank representations of the
form (\ref{eq:recomp_weight_full}) at our disposal for
$G_{t^+ c_1^+},\ldots,G_{t^+ c_\gamma^+}$, and applying (\ref{eq:nested}) yields
\begin{align*}
  G_{tc}
  &= \begin{pmatrix}
       H_{tc} &
       V_{tc_1}|_{\hat t'\times k} \widehat{Z}_{tc_1}^* P_{tc_1}^* &
       \cdots &
       V_{t^+ c_\gamma^+}|_{\hat t\times k} \widehat{Z}_{t^+ c_\gamma^+}^*
            P_{t^+ c_\gamma^+}^*
     \end{pmatrix}\\
  &= V_{tc}
     \begin{pmatrix}
       \widehat{Y}_{tc} P_{Y,tc}^* &
       E_{t c_1^+} \widehat{Z}_{t^+ c_1^+}^* P_{t^+ c_1^+}^* &
       \cdots &
       E_{t c_\gamma^+} \widehat{Z}_{t^+ c_\gamma^+}^* P_{t^+ c_\gamma^+}^*
     \end{pmatrix}\\
  &= V_{tc}
     \underbrace{\begin{pmatrix}
       \widehat{Y}_{tc} &
       E_{t c_1^+} \widehat{Z}_{t^+ c_1^+}^* &
       \cdots &
       E_{t c_\iota^+} \widehat{Z}_{t^+ c_\gamma^+}^*
     \end{pmatrix}}_{=:Z_{tc}}
     \begin{pmatrix}
       P_{Y,tc}^* & & &\\
       & P_{t^+ c_1^+}^* & & \\
       & & \ddots & \\
       & & & P_{t^+ c_\gamma^+}^*
     \end{pmatrix}.
\end{align*}
We compute a skinny QR factorization
\begin{equation*}
  \widehat{P}_{tc} \widehat{Z}_{tc} = Z_{tc}^*
\end{equation*}
and find
\begin{equation*}
  G_{tc}
  = V_{tc}
    \widehat{Z}_{tc}^*
    P_{tc}^*
  \qquad\text{ with }\qquad
  P_{tc} := \begin{pmatrix}
    P_{Y,tc} & & & \\
    & P_{t^+ c_1^+} & & \\
    & & \ddots & \\
    & & & P_{t^+ c_\gamma^+}
  \end{pmatrix} \widehat{P}_{tc}.
\end{equation*}
As a product of two isometric matrices, $P_{tc}$ is again isometric,
and since $Z_{tc}$ has only $k$ rows, $\widehat{Z}_{tc}$ is a
$k\times k$ matrix.
It is important to note that we do not need the matrices
$P_{t^+ c_\iota^+}$ to compute $\widehat{Z}_{tc}$, we can carry out
the entire algorithm without storing any of the isometric matrices.

If $t$ is the root cluster, it has no parent $t^+$, but we can
still proceed as before by setting $\gamma=0$, i.e., without contributions
inherited from the ancestors.

Once we have the \emph{total weight matrices}
$\widehat{Z}_{tc}\in\bbbc^{k\times k}$ at our disposal, we can consider
the construction of the basis.
Since $V_{tc}$ is already the name of the original basis, we use
$Q_{tc}$ for the new one.
The transfer matrices for $Q_{tc}$ are denoted by $F_{t'c}$.

If $t$ is a leaf, we have to compute the left singular vectors and
singular values of the matrix
\begin{equation*}
  G_{tc} = V_{tc} \widehat{Z}_{tc}^* P_{tc}^*,
\end{equation*}
and this is equivalent to computing these quantities only for the
thin matrix $V_{tc} \widehat{Z}_{tc}^*$.
We choose a rank $k_{tc}$ for the new basis and use the first
$k_{tc}$ left singular vectors as columns of the new basis
matrix $Q_{tc}\in\bbbc^{\hat t\times k_{tc}}$.

If $t$ is not a leaf, we assume again $\sons(t)=\{t_1,t_2\}$,
let $c_1 := \sd{t_1}(c)$, $c_2 := \sd{t_2}(c)$, and have to compute
the left singular vectors and singular values of the matrix
\begin{align*}
  \widehat{G}_{tc}
  &= \begin{pmatrix}
       Q_{t_1 c_1}^* & \\
       & Q_{t_2 c_2}^*
     \end{pmatrix} G_{tc}
   = \begin{pmatrix}
       Q_{t_1 c_1}^* & \\
       & Q_{t_2 c_2}^*
     \end{pmatrix}
     V_{tc} \widehat{Z}_{tc}^* P_{tc}^*\\
  &= \begin{pmatrix}
       Q_{t_1 c_1}^* & \\
       & Q_{t_2 c_2}^*
     \end{pmatrix}
     \begin{pmatrix}
       V_{t_1 c_1} E_{t_1 c}\\
       V_{t_2 c_2} E_{t_2 c}
     \end{pmatrix} \widehat{Z}_{tc}^* P_{tc}^*
  = \begin{pmatrix}
      Q_{t_1 c_1}^* V_{t_1 c_1} E_{t_1 c}\\
      Q_{t_2 c_2}^* V_{t_2 c_2} E_{t_2 c}
    \end{pmatrix} \widehat{Z}_{tc}^* P_{tc}^*.
\end{align*}
In order to prepare this matrix efficiently, we introduce the matrices
\begin{align}\label{eq:C_tc}
  C_{tc} &:= Q_{tc}^* V_{tc} &
  &\text{ for all } t\in\ctI,\ c\in\mathcal{D}_t,
\end{align}
that describe the change of basis from $V_{tc}$ to $Q_{tc}$.
With these matrices, we have
\begin{equation*}
  \widehat{G}_{tc}
  = \underbrace{\begin{pmatrix}
      C_{t_1 c_1} E_{t_1 c}\\
      C_{t_2 c_2} E_{t_2 c}
    \end{pmatrix}}_{=:\widehat{V}_{tc}} \widehat{Z}_{tc}^* P_{tc}^*
\end{equation*}
and only have to compute the SVD of $\widehat{V}_{tc} \widehat{Z}_{tc}^*$,
choose a rank $k_{tc}$, and use the first $k_{tc}$ left singular
vectors as columns of the matrix $\widehat{Q}_{tc}$ that can
be split into
\begin{equation*}
  \widehat{Q}_{tc}
  = \begin{pmatrix}
      F_{t_1 c}\\ F_{t_2 c}
    \end{pmatrix}
\end{equation*}
to obtain the transfer matrices for the new cluster basis.
In this case, we can use $C_{tc} = \widehat{Q}_{tc}^* \widehat{V}_{tc}$
to compute the basis-change matrix efficiently.

%
%

\section{Complexity}

In order to analyze the complexity of the new algorithms, we follow
the approach of \cite[Section~5]{BO15}:
for the sake of simplicity, we assume that all bounding boxes on
the same level are identical up to translation.
We also assume that the cluster tree is geometrically regular and
that the surface $\Omega$ is two-dimensional, i.e., that there
are constants $\Csb,\Csn,\Cbp,\Cbb,\Cov,\Crs,\Cun\in\bbbr_{>0}$
such that
\begin{align*}
  \diam(B_t) &\leq \Csb \diam(B_{t'}) &
  &\text{ for all } t\in\ctI,\ t'\in\sons(t),\\
  \#\sons(t) &\leq \Csn,\quad \#\sons(t)\neq 1 &
  &\text{ for all } t\in\ctI,\\
  |\Omega\cap\mathcal{B}(x,r)| &\leq \Cbp r^2 &
  &\text{ for all } x\in\bbbr^3,\ r\in\bbbr_{\geq 0},\\
  \diam^2(B_t) &\leq \Cbb |B_t\cap\Omega| &
  &\text{ for all } t\in\ctI,\\
  \#\{ t\in\ctIl{\ell}\ :\ x\in B_t \} &\leq \Cov &
  &\text{ for all } x\in\Omega,\ \ell\in[0:\pI],\\
  \Clf \kappa \diam(B_t) &\leq 1 &
  &\text{ for all leaves } t\in\lfI,\\
  \Crs^{-1} k &\leq \#\hat t \leq \Crs k &
  &\text{ for all leaves } t\in\lfI,\\
  \eta_2 \dist(B_t, B_s) < \diam(B_t)
  &\Rightarrow \#\hat s\leq \Cun \#\hat t &
  &\text{ for all } t\in\lfI,\ s\in\ctI\\
  & & & \quad\text{ with } \level(t)=\level(s).
\end{align*}
The constant $\Csb$ ensures that child clusters do
not decreases in size too quickly, while $\Csn$ provides an upper
bound for the number of children.
$\Cbp$ and $\Cbb$ measure how ``convoluted'' the surface $\Omega$
is, $\Cov$ describes the overlap of clusters.
$\Clf$ and $\Crs$ ensure that the leaves are small enough compared
to the wavelength, and $\Cun$ can be interpreted as a quasi-uniformity
condition for neighbouring leaf clusters.
Additionally we assume that the number of directions associated with a
cluster is bounded, i.e., that there is a constant $\Cdi\in\bbbr_{>0}$ with
\begin{align}\label{eq:directions_cluster}
  \#\mathcal{D}_t &\leq \Cdi (1 + \kappa^2 \diam^2(B_t)) &
  &\text{ for all } t\in\ctI.
\end{align}
If the directions are constructed as in
Remark~\ref{re:construction_of_directions}, this condition is
satisfied.

According to \cite[Lemma~8]{BO15}, there is a \emph{sparsity constant}
$\Csp \in \bbbr_{>0}$ such that
\begin{align}\label{eq:sparsity_bound}
 \sum_{c \in \mathcal{D}_{t}} \# \brow(t,c) \leq
\begin{cases}
C_{sp} & \text{if } C_{cp} \kappa \diam(B_t) < 1\\
C_{sp} \kappa^{2} \diam(B_{t})^2 & otherwise
\end{cases}
\end{align}
holds for all $t\in\ctI$.
We introduce the short notation
\begin{align*}
  \Cspt &:= \begin{cases}
              \Csp & \text{ if } \Csb \kappa \diam(B_t) < 1\\
              \Csp \kappa^2 \diam(B_t)^2 & \text{ otherwise}
            \end{cases} &
  &\text{ for all } t\in\ctI.
\end{align*}
According to \cite[Lemma~9]{BO15}, there is a constant $\Clv\in\bbbr_{>0}$
such that
\begin{subequations}\label{eq:cluster_bounds}
\begin{align}
  \#\ctIl{\ell} &\leq \Clv \frac{|\Omega|}{\diam^2(B_t)} &
  &\text{ for all } \ell\in[0:\pI],\ t\in\ctIl{\ell},
     \label{eq:level_bound}\\
  \#\ctI &\leq \Clv \frac{\#\Idx}{k}.
     \label{eq:clusters_bound}
\end{align}
\end{subequations}
The estimates (\ref{eq:sparsity_bound}) and (\ref{eq:cluster_bounds})
give rise to the following fundamental result.

%
%
\begin{lemma}[Block and cluster sums]
\label{le:sums}
There are constants $\Cbs,\Ccs\in\bbbr_{>0}$ with
\begin{subequations}
\begin{align}
  \sum_{t\in\ctI} \sum_{c\in\mathcal{D}_t} \#\brow(t,c)
  &\leq \Cbs \left( \#\ctI + \Clv (\pI+1) \kappa^2 \right),
     \label{eq:block_sum}\\
  \sum_{t\in\ctI} \#\mathcal{D}_t
  &\leq \Ccs \left( \#\ctI + \Clv (\pI+1) \kappa^2 \right).
     \label{eq:cluster_sum}
\end{align}
\end{subequations}
\end{lemma}
\begin{proof}
Combining (\ref{eq:sparsity_bound}) and (\ref{eq:level_bound}) yields
\begin{align*}
  \sum_{t\in\ctI} \sum_{c\in\mathcal{D}_t} \#\brow(t,c)
  &= \sum_{\substack{t\in\ctI\\ \Csb\kappa\diam(B_t) < 1}}
     \sum_{c\in\mathcal{D}_t} \#\brow(t,c)
   + \kern -20pt \sum_{\substack{t\in\ctI\\ \Csb\kappa\diam(B_t) \geq 1}}
     \sum_{c\in\mathcal{D}_t} \#\brow(t,c)\\
  &\leq \sum_{\substack{t\in\ctI\\ \Csb\kappa\diam(B_t) < 1}} \Csp
    + \sum_{\substack{t\in\ctI\\ \Csb\kappa\diam(B_t) \geq 1}}
        \Csp \kappa^2 \diam^2(B_t)\\
  &\leq \Csp \#\ctI
    + \sum_{\ell=0}^{\pI} \sum_{t\in\ctIl{\ell}}
        \Csp \kappa^2 \diam^2(B_t)\\
  &\leq \Csp \#\ctI + \sum_{\ell=0}^{\pI} 
        \Clv \frac{|\Omega|}{\diam^2(B_t)} \Csp \kappa^2 \diam^2(B_t)\\
  &\leq \Csp \#\ctI + \Clv \Csp |\Omega| (\pI+1)  \kappa^2,
\end{align*}
and we obtain (\ref{eq:block_sum}) by choosing
$\Cbs := \max\{ \Csp, \Csp |\Omega| \}$.

For the second estimate, we combine (\ref{eq:directions_cluster})
with (\ref{eq:level_bound}) to find
\begin{align*}
  \sum_{t\in\ctI} \#\mathcal{D}_t
  &\leq \Cdi \sum_{t\in\ctI} (1 + \kappa^2 \diam^2(B_t))
   = \Cdi \#\ctI
      + \Cdi \sum_{\ell=0}^{\pI} \sum_{t\in\ctIl{\ell}} \kappa^2
          \diam^2(B_t)\\
  &\leq \Cdi \#\ctI
      + \Cdi \sum_{\ell=0}^{\pI} \Clv \frac{|\Omega|}{\diam^2(B_t)}
          \kappa^2 \diam^2(B_t)\\
  &= \Cdi \#\ctI + \Cdi (\pI+1) \Clv |\Omega| \kappa^2,
\end{align*}
and we can obtain (\ref{eq:cluster_sum}) by choosing
$\Ccs := \max\{ \Cdi, \Cdi |\Omega| \}$.
\end{proof}

To establish an estimate for the complexity we need to bound the work
of the QR factorization as well as the SVD.
We assume that there are constants $\Cqr,\Csvd$ such that the work
of computing the QR factorization and the SVD of a matrix
$A\in\bbbc^{m\times n}$ up to machine accuracy is bounded by
\begin{subequations}
\begin{align}
  &\Cqr m n \min\{m,n\},\label{eq:qrbound}\\
  &\Csvd m n \min\{m,n\},\label{eq:svdbound}
\end{align}
\end{subequations}
respectively.

Now we can consider the complexity of the different phases of the
recompression algorithm.
We first have to compute the \emph{basis weight matrices} $R_{W,tc}$
for the original cluster basis $(W_{tc})_{t\in\ctI,c\in\mathcal{D}_t}$.

%
%
\begin{lemma}[Basis weights]
There is a constant $\Cbw\in\bbbr_{>0}$ such that computing the basis
weights $(R_{W,tc})_{t\in\ctI,c\in\mathcal{D}_t}$, cf.
(\ref{eq:basis_weights_def}), requires not more than
\begin{equation*}
  \Cbw k^3 \left( \frac{\#\Idx}{k} + (\pI+1) \kappa^2 \right)
  \text{ operations.}
\end{equation*}
\end{lemma}
\begin{proof}
Using \cite[Algorithm~16]{BO10}, adapted for multiple directions per
cluster, this task takes $(\Cqr+2) k^3$ operations per cluster
and direction, and Lemma~\ref{le:sums} together with (\ref{eq:clusters_bound})
yields
\begin{align*}
  \sum_{t\in\ctI} \sum_{c\in\mathcal{D}_t} (\Cqr+2) k^3
  &\leq (\Cqr+2) k^3 \Ccs (\#\ctI + \Clv (\pI+1) \kappa^2)\\
  &= \Ccs \Clv (\Cqr+2) k^3 \left(\frac{\#\Idx}{k}
                     + (\pI+1) \kappa^2\right).
\end{align*}
We let $\Cbw := \Ccs \Clv (\Cqr+2)$ to complete the proof.
\end{proof}

The second step is to compute the \emph{total weight matrices}
$\widehat{Z}_{tc}$ for the original cluster basis
$(V_{tc})_{t\in\ctI,c\in\mathcal{D}_t}$.

%
%
\begin{lemma}[Total weights]
There is a constant $\Cwe\in\bbbr_{>0}$ such that computing the total
weights $(\widehat{Z}_{tc})_{t\in\ctI,c\in\mathcal{D}_t}$ requires not more than
\begin{equation*}
  \Cwe k^3 \left( \frac{\#\Idx}{k} + (\pI+1) \kappa^2 \right)
  \text{ operations.}
\end{equation*}
\end{lemma}
\begin{proof}
Let $t\in\ctI$ and $c\in\mathcal{D}_t$.

We have to set up the matrix $Z_{tc}$.
For all $s\in\brow(t,c)$, this means computing the product
$S_{t s c} R_{W,sc}^*$, which takes not more than $2 k^3$ operations.

If there is a corresponding parent cluster $t^+$, we also have to compute
the product of the transfer matrix $E_{t c^{+}}$ and the parent's weight
$\widehat{Z}_{t^{+}c^{+}}$ for all $c^+\in\sd{t^+}^{-1}(\{c\})$, which takes
not more than $2 k^3$ operations per product.

We denote the number of columns of $Z_{tc}$ by
\begin{equation*}
  m := k \#\sd{t^+}^{-1}(\{c\}) + k (\#\brow(t,c))
\end{equation*}
and have shown that $2 m k^2$ operations are needed to set up this
matrix.

Now follows a QR factorization of the matrix $Z_{tc}^*\in\bbbc^{m\times k}$, which
requires not more than $\Cqr m k \min\{m,k\} \leq \Cqr m k^2$ operations.

In consequence, the complexity for the whole cluster tree is bounded by
\begin{equation*}
  \sum_{t\in\ctI} \sum_{c\in\mathcal{D}_t} (\Cqr+2) m k^2
  = (\Cqr+2) k^3 \sum_{t\in\ctI} \sum_{c\in\mathcal{D}_t}
        \#\sd{t^+}^{-1}(\{c\}) + \#\brow(t,c).
\end{equation*}
Since $\sd{t^+}$ maps every direction $c^+\in\mathcal{D}_{t^+}$ to a
direction $c=\sd{t^+}(c^+)\in\mathcal{D}_t$, we have
\begin{align*}
  \mathcal{D}_{t^+} &= \bigcup_{c\in\mathcal{D}_t} \sd{t^+}^{-1}(\{c\}), &
  \#\mathcal{D}_{t^+} &= \sum_{c\in\mathcal{D}_t} \#\sd{t^+}^{-1}(\{c\})
\end{align*}
and can use Lemma~\ref{le:sums} (and the convention
$\mathcal{D}_{t^+}=\emptyset$ if $t$ is the root) to find the bound
\begin{align*}
  \sum_{t\in\ctI} \sum_{c\in\mathcal{D}_t} (\Cqr+2) m k^2
  &= (\Cqr+2) k^3 \sum_{t\in\ctI} \sum_{c\in\mathcal{D}_t}
         \#\sd{t^+}^{-1}(\{c\}) + \#\brow(t,c)\\
  &= (\Cqr+2) k^3 \sum_{t\in\ctI} \#\mathcal{D}_{t^+}
         + \sum_{c\in\mathcal{D}_t} \#\brow(t,c)\\
  &= (\Cqr+2) k^3 \sum_{t^+\in\ctI} \sum_{t\in\sons(t^+)} \#\mathcal{D}_{t^+}\\
  &\qquad + (\Cqr+2) k^3 \sum_{t\in\ctI} \sum_{c\in\mathcal{D}_t}
         \#\brow(t,c)\\
  &\leq (\Cqr+2) k^3 \sum_{t^+\in\ctI} \Csn \#\mathcal{D}_{t^+}\\
  &\qquad + (\Cqr+2) k^3 \sum_{t\in\ctI} \sum_{c\in\mathcal{D}_t}
         \#\brow(t,c)\\
  &\leq (\Cqr+2) \Csn k^3 \Ccs (\#\ctI + \Clv (\pI+1) \kappa^2)\\
  &\qquad + (\Cqr+2) k^3 \Cbs (\#\ctI + \Clv (\pI+1) \kappa^2)\\
  &= (\Cqr+2) (\Csn \Ccs + \Cbs) k^3 (\#\ctI + \Clv (\pI+1) \kappa^2).
\end{align*}
We can use (\ref{eq:clusters_bound}) to complete the proof
with $\Cwe := (\Cqr+2) (\Csn \Ccs + \Cbs) \Clv$.
\end{proof}

Now we can address the construction of the improved cluster basis.

%
%
\begin{lemma}[Truncation]\label{lem:comptrun}
There is a constant $\Ctr\in\bbbr_{>0}$ such that computing the
improved cluster basis $(Q_{tc})_{t\in\ctI,c\in\mathcal{D}_t}$ requires
not more than
\begin{equation*}
  \Ctr k^3 \left( \frac{\#\Idx}{k} + (\pI + 1) \kappa^2 \right)
  \text{ operations.}
\end{equation*}
\end{lemma}
\begin{proof}
Let $t\in\ctI$ and $c\in\mathcal{D}_t$.

If $t$ is a leaf, $V_{tc}\in\bbbc^{m\times k}$, $m:=\#\hat t$, is used directly.
We compute the product $V_{tc} \widehat{Z}_{tc}^*\in\bbbc^{m\times k}$ in not
more than $2 m k^2$ operations, its SVD in not more than $\Csvd m k \min\{m,k\}
\leq \Csvd m k^2$ operations, and the basis-change matrix
$C_{tc}$ in not more than $2 m k^2$ operations.

Due to our assumptions, we have $m = \#\hat t \leq \Crs k$, and the number
of operations for leaf clusters is bounded by $(\Csvd+4) \Crs k^3$.

If $t$ is not a leaf, we compute the product of the transfer matrix $E_{t'c}$
and the already calculated basis-change matrix $C_{t'c'}$ for every
child $t'\in\sons(t)$ and $c' = \sd{t'}(c)$, and the resulting matrix
$\widehat{V}_{tc}$ has $m := \sum\limits_{t' \in \sons(t)} k_{t'c'}$ rows
and $k$ columns.
Computing all products takes not more than
\begin{equation*}
  \sum_{t' \in \sons(t)} 2 k^{2} k_{t'c'} = 2 m k^2
  \text{ operations.}
\end{equation*}
Now the matrix $V_{tc} \widehat{Z}_{tc}^*$ has to be computed, this takes
not more than $2 m k^2$ operations.
Due to (\ref{eq:svdbound}), its SVD can be computed in
$\Csvd m k \min\{m,k\} \leq \Csvd m k^2$ operations.
Finally the basis-change matrix $C_{tc}$ can be computed in not more
than $2 m k^2$ operations.

Due to our assumptions, $\#\sons(t)\leq\Csn$ holds and we have
$m\leq \Csn k$, so the number of operations for non-leaf clusters
is bounded by $(\Csvd+6) \Csn k^3$.

Finding the correct ranks $k_{tc}$ requires the inspection of $m$ singular
values and can be accomplished in $\mathcal{O}(k)$ operations, so we can
conclude that there is a constant $C$ such that not more than $C k^3$
operations are required per cluster $t\in\ctI$ and direction
$c\in\mathcal{D}_t$.

The total number of operations is bounded by
\begin{equation*}
  \sum_{t\in\ctI} \sum_{c\in\mathcal{D}_t} C k^3
  = C k^3 \sum_{t\in\ctI} \#\mathcal{D}_t
  \leq C \Ccs k^3 (\#\ctI + \Clv (\pI+1) \kappa^2)
\end{equation*}
due to (\ref{eq:cluster_sum}), and (\ref{eq:clusters_bound})
completes the proof.
\end{proof}

The only thing left is the calculation of the new coupling matrices, but
this is a simple matrix multiplication of the old coupling matrices
 with the basis-change matrices (\ref{eq:C_tc}).

%
%
\begin{lemma}[Projections]
There is a constant $\Cpr\in\bbbr_{>0}$ such that computing the
new coupling matrices $(\widetilde{S}_{b})_{b\in\lfaII}$ requires
not more than
\begin{equation*}
  \Cpr k^3 \left( \frac{\#\Idx}{k} + (\pI+1) \kappa^2 \right)
  \text{ operations.}
\end{equation*}
\end{lemma}
\begin{proof}
Computing the products $T_b := C_{tc} S_b$ and $\widetilde{S}_b
:= T_b C_{sc}^*$ requires not more than $4 k^3$ operations for
each block $b\in\lfaII$.
Due to (\ref{eq:block_sum}), the total number of operations
is bounded by
\begin{equation*}
  \sum_{b\in\lfaII} 4 k^3
  \leq 4 k^3 \sum_{t\in\ctI} \sum_{c\in\mathcal{D}_t} \#\brow(t,c)
  \leq 4 \Cbs k^3 (\#\ctI + \Clv (\pI+1) \kappa^2),
\end{equation*}
and we can use (\ref{eq:clusters_bound}) to obtain our estimate
with $\Cpr := 4 \Cbs \Clv$.
\end{proof}

For the complete recompression, we have to compute the basis
and total weights for the row and the column cluster basis,
we have to truncate both bases, and we have to apply the
projection to obtain the improved $\mathcal{DH}^2$-matrix
representation.

%
%
\begin{theorem}[Complexity]
\label{th:complexity}
Let a $\mathcal{DH}^{2}$-matrix be given.
The entire recompression algorithm requires not more than
\begin{equation*}
  (2\Cbw + 2\Cwe + 2\Ctr + \Cpr) k^3
  \left( \frac{\# \Idx}{k} + (\pI+1) \kappa^{2} \right)
  \text{ operations.}
\end{equation*}
\end{theorem}
\begin{proof}
The proof follows by simply adding the estimates provided by
the previous lemmas.
\end{proof}

%
%
\begin{remark}[Complexity]
\label{re:complexity}
Let $n := \#\Idx$ denote the matrix dimension.
Since $\Crs$ should be bounded independently of $n$, we have
to expect $\pI\sim\log n$.

If the wave number $\kappa$ is constant, the first term in the
estimate of Theorem~\ref{th:complexity} is dominant and the recompression
algorithm requires $\mathcal{O}(n k^2)$ operations.

In the high-frequency case, we have $\kappa^2 \sim n$, the second
term is dominant, and the recompression algorithm requires
$\mathcal{O}(n k^3 \log n)$ operations.
\end{remark}

%
%

\section{Numerical experiments}

As an example we use the three-dimensional unit sphere
and cube.
For the cube, we simply represent each face by two triangles that are
then regularly refined.
For the sphere, we start with the double pyramid
$P=\{x\in\bbbr^3\ :\ |x_1|+|x_2|+|x_3|=1\}$, refine every one of its
eight faces regularly, and shift the resulting vertices to the unit sphere.
For constructing the Galerkin stiffness matrix
$G \in \bbbc^{ \Idx \times  \Idx}$ we use piecewise constant basis functions
and Sauter-Erichsen-Schwab quadrature of order $n_q = 5$
\cite{SASC11, ERSA98}
for triangles that share a vertex, an edge, or are identical, and
otherwise Gau\ss{} quadrature with Duffy transformation of order
$n_q = 3$ \cite{DU82}.

As clustering strategy the standard binary space partitioning is applied
until clusters contain not more than $32$ elements.
We used $\eta_1 = 10$ for creating the directions (\ref{eq:adm_directions}),
and $\eta_2=1$ for the standard (\ref{eq:adm_standard}) and parabolic
admissibility condition (\ref{eq:adm_parabolic}).
The initial $\mathcal{DH}^2$-matrix approximation is constructed by
directional interpolation of order $m=4$, and the initial rank is
$k=4^3=64$.

We choose the wave number in such a way, that $\kappa h \approx 0.6$
is ensured, i.e., we have approximately ten elements per wavelength.
For the recompression algorithm we employ an accuracy $\epsilon = 10^{-4}$
for the block-wise relative Frobenius norm.

We have parallelized the algorithm using the OpenMP standard
in order to take advantage of multicore shared-memory systems:
since the basis weights (\ref{eq:basis_weights_def}) are computed
by a bottom-up recursion, all weights within one level of the
tree can be computed in parallel.
Since the total weights (\ref{eq:recomp_weight_full}) are computed
by a top-down recursion, we can also perform all operations within
the same level in parallel.
Finally, the construction of the adaptive cluster bases is again
a bottom-up procedure and therefore accessible to the same
approach.
On a server with two Intel\textsuperscript{\textregistered}
Xeon\textsuperscript{\textregistered} Platinum 8160 processors
with a total of 48 cores, direct interpolation for
$131\,072$ triangles takes $203$ seconds with the parallel
implementation and $5\,543$ seconds without, a speedup of $27$,
while recompression takes $1\,475$ seconds with the parallel
version and $38\,676$ seconds without, a speedup of $26$.
We have observed that the speedup improves as the problem
size grows.

Table \ref{tab:slp_frob} shows our results for the single layer potential
on the unit sphere.
The first column gives the number $n$ of degrees of freedom, the second
the wave number $\kappa$.
The third and fourth column show the storage per degree of
freedom in KB (1 KB is 1\,024 bytes) of the original cluster basis and
the original $\mathcal{DH}^2$-matrix, the fifth, sixth and sevenths
colums give the rank, storage for the basis, and storage for the 
matrix after the recompression.
The last two columns give the absolute and relative errors
between the $\mathcal{DH}^{2}$-matrix and its recompressed version, measured
in the Frobenius norm.

%
%

\begin{table}[ht]
\begin{equation*}
\begin{array}{r|r|r|r|r|r|r|r|r} \label{tab:slp_frob}
& & \multicolumn{2}{c|}{\text{original}}
  & \multicolumn{3}{c|}{\text{recomp.}} & \\
n & \kappa & \text{basis}& \text{matrix} & \text{rank} & \text{basis} 
  & \text{matrix} & \text{error} & \text{rel. error}\\
\hline
2\,352 & 3.5 & 0.6 & 51.6 & 15 & 0.1 & 36.1 & 4.93_{-8} & 2.11_{-6}\\
\hline
4\,800 & 5 & 2.0 & 166.4 & 20 & 0.3 & 66.6 & 3.69_{-8} & 3.12_{-6}\\
\hline
9\,408 & 7 & 3.2 & 326.5 & 23 & 0.4 & 116.5 & 1.70_{-8} & 2.74_{-6}\\
\hline
19\,200 & 10 & 6.4 & 496.7 & 28 & 0.6 & 167.9 & 1.06_{-8} & 3.39_{-6}\\
\hline
37\,632 & 14 & 13.8 & 941.3 & 30 & 1.9 & 274.8 & 6.57_{-9} & 4.03_{-6}\\
\hline
76\,800 & 20 & 25.2 & 1\,516.5 & 33 & 1.4 & 374.0 & 3.80_{-9} & 4.64_{-6}\\
\hline
150\,528 & 28 & 35.3 & 2\,183.7 & 33 & 2.0 & 484.5 & 2.13_{-9} & 4.98_{-6}\\
\hline 
307\,200 & 40 & 57.8 & 2\,730.7 & 36 & 3.1 & 607.1 & 1.28_{-9} & 5.99_{-6}
\end{array}
\end{equation*}
\caption{Single layer potential operator on the cube (Frobenius norm)}
\end{table}

Similar results are obtained for the double layer potential, they are
presented with the same structure as above in Table~\ref{tab:dlp_frob}.


\begin{table}[ht]
\begin{equation*}
\begin{array}{r|r|r|r|r|r|r|r|r} \label{tab:dlp_frob}
& & \multicolumn{2}{c|}{\text{original}}
  & \multicolumn{3}{c|}{\text{recomp.}} & \\
n & \kappa & \text{basis} & \text{matrix} & \text{rank} & \text{basis}
  & \text{matrix} & \text{error} & \text{rel. error}\\
\hline
2\,352 & 3.5 & 0.6 & 51.6 & 18 & 0.1 & 36.3 & 1.21_{-7} & 2.45_{-7}\\
\hline
4\,800 & 5 & 2.0 & 166.4 & 22 & 0.3 & 67.5 & 1.30_{-7} & 3.75_{-7}\\
\hline
9\,408 & 7 & 3.2 & 326.5 & 26 & 0.4 & 118.3 & 8.08_{-8} & 3.26_{-7}\\
\hline
19\,200 & 10 & 6.4 & 496.7 & 30 & 0.8 & 176.5 & 1.97_{-7} & 1.13_{-6}\\
\hline
37\,632 & 14 & 13.8 & 941.3 & 33 & 1.5 & 294.8 & 2.19_{-7} & 1.76_{-6}\\
\hline
76\,800 & 20 & 25.2 & 1\,516.5 & 37 & 2.4 & 409.5 & 1.88_{-7} & 2.16_{-6}\\
\hline
150\,528 & 28 & 35.3 & 2\,183.7 & 37 & 3.6 & 541.4 & 1.46_{-7} & 2.36_{-6}\\
\hline 
307\,200 & 40 & 57.8 & 2\,730.7 & 43 & 4.4 & 647.9 & 6.33_{-8} & 1.46_{-6}
\end{array}
\end{equation*}
\caption{Double layer potential operator on the cube (Frobenius norm)}
\end{table}

\begin{figure}[ht]
\begin{center}\label{fig:DLP_setup}
\subfigure[Memory (dlp)]{\includegraphics[scale = 0.45]{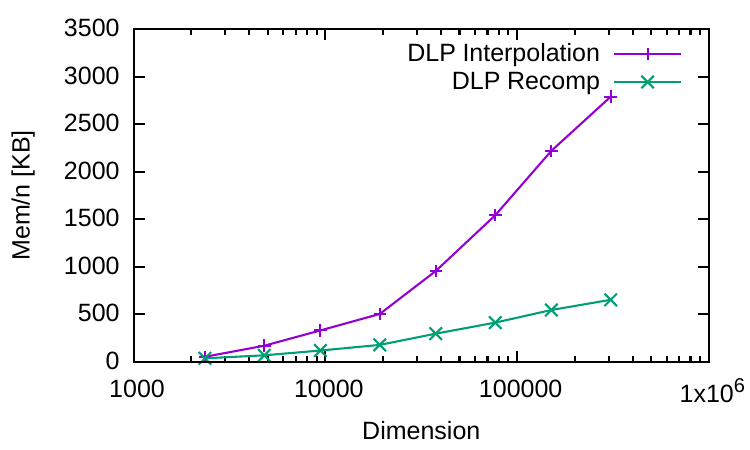}} \quad
\subfigure[Complexity (dlp)]{\includegraphics[scale = 0.45]{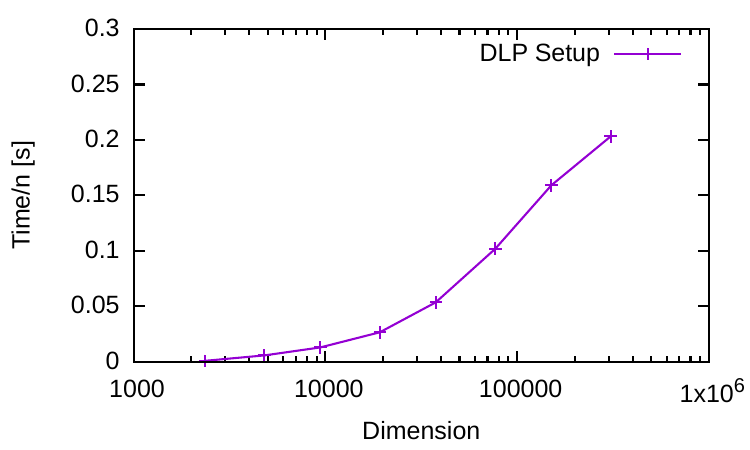}}

\subfigure[Memory (slp)]{\includegraphics[scale = 0.45]{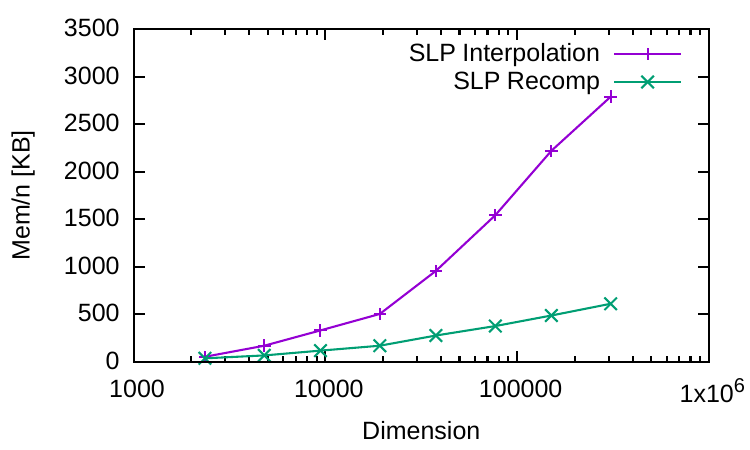} } \quad
\subfigure[Complexity (slp)]{\includegraphics[scale = 0.45]{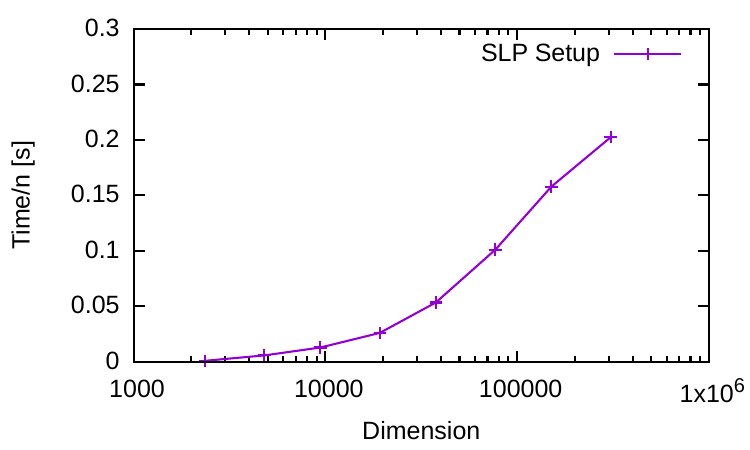} }
\caption{Memory and time per degree of freedom for the cube}
\end{center}
\end{figure}

To outline the results, Figure~\ref{fig:DLP_setup} shows the memory
requirements per degree of freedom as function of $n$ for the single
layer (c) and the double layer potential (a).
In both cases from the beginning of our experiment the recompressed
version needs less storage and the memory advantage improves with $n$.

The panels (b) and (d) in Figure~\ref{fig:DLP_setup} show the runtime
per degree of freedom for the recompression of the $\mathcal{DH}^2$-matrix
obtained by interpolation.
Since we use a logarithmic scale for $n$ and a linear scale for the
time divided by $n$, the experiment confirms our expectation of
$\mathcal{O}(n \log n)$ complexity.


Even for higher wave numbers and other norms the algorithm keeps
this behavior:
Table \ref{tab:slp_eucl} shows results for doubled wave numbers on
the unit sphere, where the error control strategy for the recompression
uses the spectral norm instead of the Frobenius norm.

%
\begin{table}[ht]
\begin{equation*}
\begin{array}{r|r|r|r|r|r|r|r|r} \label{tab:slp_eucl}
& & \multicolumn{2}{c|}{\text{original}}
  & \multicolumn{3}{c|}{\text{recomp.}} & \\
n & \kappa & \text{basis} & \text{matrix} & \text{rank} & \text{basis}
  & \text{matrix} & \text{error} & \text{rel. error}\\
\hline
2\,048 & 8 & 0.3 & 35.2 & 8 & 0.0 & 32.3	& 3.94_{-9\phantom{0}} & 2.74_{-6}\\
\hline
4\,608 & 12 & 1.3 & 82.9 & 9 & 0.1 & 71.3 & 3.22_{-9\phantom{0}} & 6.65_{-6}\\
\hline
8\,192 & 16 & 4.9 & 205.5 & 12 & 0.2 & 120.6 & 1.63_{-9\phantom{0}} & 7.37_{-6}\\
\hline
18\,432 & 24 & 10.1 & 821.9 & 15 & 0.6 & 223.0 & 1.26_{-9\phantom{0}} & 1.66_{-5}\\
\hline
32\,768 & 32 & 45.7 & 1\,450.8 & 15 & 1.7 & 312.3 & 6.15_{-10} & 1.74_{-5}\\
\hline
73\,728 & 48 & 72.7 & 3\,336.2 & 18 & 2.8 & 459.3 & 2.62_{-10} & 2.16_{-5}\\
\hline
131\,072 & 64 & 98.4 & 5\,114.2 & 21 & 3.9 & 586.7 & 1.49_{-10} & 2.64_{-5} \\
\end{array}
\end{equation*}
\caption{Single layer potential operator on the sphere (spectral norm)}
\end{table}

Next we consider a more realistic geometry:
a mesh of an airplane, more precisely a Boeing 747, comprised
of $549\,836$ triangles and $274\,920$ vertices, provided by courtesy
of Boris Dilba.
We use the wave number $\kappa = 3.15$, corresponding to
a wavelength of approximately $2$.
The maximal extent of the airplane is approximately $62$, i.e.,
approximately $31$ wavelengths.
A picture of our object of study is shown in
Figure~\ref{fig:boeing}.

\begin{figure}[ht]
\begin{center}
\includegraphics[scale=0.25]{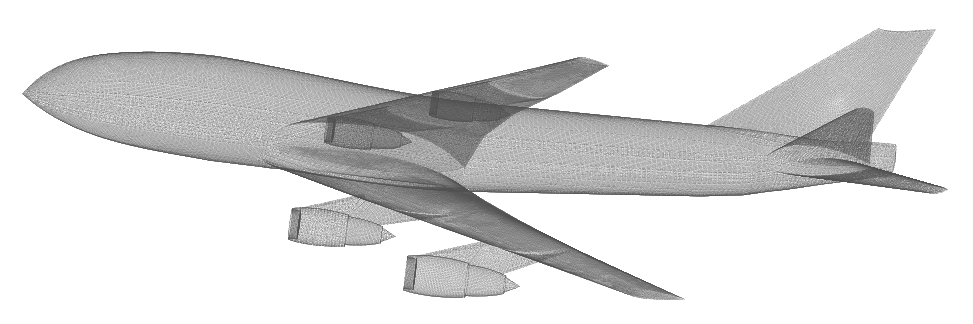}
\end{center}
\caption{Mesh of a Boeing 747.}\label{fig:boeing}
\end{figure}

We have modified our recompression algorithm such that
it can be applied during the set-up process to reduce
intermediate storage requirements:
the cluster basis is orthogonalized immediately, and the coupling matrices
are constructed on the fly when needed during the recompression
algorithm.
With the modified algorithm we are able to set up the
$\mathcal{DH}^{2}$-matrix with linear basis functions for the
airplane mesh to obtain the results given in
Table~\ref{tab:boeing}, where we have varied both the recompression
tolerance $\epsilon$ and the interpolation order $m$.
We also report run-times (in hours) measured on our shared-memory system.


\begin{table}[h]
\begin{equation*}
\begin{array}{r|r|r|r|r|r|r|r} \label{tab:boeing}
\epsilon & m & \text{time} & \text{rank} & \text{basis} & \text{matrix}
  & \text{error} & \text{rel. error}\\
\hline
1.0_{-2} & 4 & 0.5 & 28 & 1.7 & 83.6 & 1.0_{-3} & 4.9_{-2}\\
\hline
1.0_{-3} & 5 & 1.3 & 38 & 2.9 & 105.9 & 1.8_{-4} & 8.4_{-4}\\
\hline
1.0_{-4} & 6 & 3.2 & 46 & 4.7 & 138.4 & 3.1_{-5} & 1.5_{-4}\\
\end{array}
\end{equation*}
\caption{Boeing 747 with single layer potential (direct recompression)}
\end{table}

We have measured the relative error between the dense,
i.e., uncompressed, matrix and the recompressed approximation in the
Frobenius norm.
To put these results in perspective, the dense matrix takes about $4\,296$~KB
per degree of freedom.

Compared to interpolation, our recompression algorithm reduces the
storage requirements for the cluster basis from $194$~KB to $1.7$~KB for
order $m=4$ and from $2\,137$~KB to $4.7$~KB for order $m=6$.
For the coupling matrices, we save similar amounts of storage:
we go from $2\,322$~KB to $83.6$~KB for order $m=4$ and from
$25\,833$~KB to $138.4$~KB for order $m=6$.
The measured relative Frobenius error is always well below the
prescribed tolerance.

We can see that $\mathcal{DH}^2$-recompression is absolutely crucial
in order to turn the initial approximation constructed by directional
interpolation into a practically useful representation that saves
approximately $96\%$ of storage at an accuracy of $3.1\times 10^{-5}$.

%
%

\bibliographystyle{plain}
\bibliography{scicomp}

\begin{thebibliography}{10}

\bibitem{BEKUVE15}
M.~Bebendorf, C.~Kuske, and R.~Venn.
\newblock Wideband nested cross approximation for {Helmholtz} problems.
\newblock {\em Num. Math.}, 130(1):1--34, 2015.

\bibitem{BO10}
S.~{B\"orm}.
\newblock {\em Efficient Numerical Methods for Non-local Operators: {${\mathcal
  H}^2$}-Matrix Compression, Algorithms and Analysis}, volume~14 of {\em EMS
  Tracts in Mathematics}.
\newblock EMS, 2010.

\bibitem{BO15}
S.~{B\"orm}.
\newblock Directional {$\mathcal{H}^2$}-matrix compression for high-frequency
  problems.
\newblock {\em Num. Lin. Alg. Appl.}, 24(6), 2017.
\newblock available online at \url{http://dx.doi.org/10.1002/nla.2112}.

\bibitem{BOME15}
S.~{B\"orm} and J.~M. Melenk.
\newblock Approximation of the high-frequency {Helmholtz} kernel by nested
  directional interpolation: error analysis.
\newblock {\em Num. Math.}, 137(1):1--34, 2017.
\newblock available at \url{http://dx.doi.org/10.1007/s00211-017-0873-y}.

\bibitem{BR91}
A.~Brandt.
\newblock Multilevel computations of integral transforms and particle
  interactions with oscillatory kernels.
\newblock {\em Comp. Phys. Comm.}, 65(1--3):24--38, 1991.

\bibitem{DU82}
M.~G. Duffy.
\newblock Quadrature over a pyramid or cube of integrands with a singularity at
  a vertex.
\newblock {\em SIAM J. Num. Anal.}, 19(6):1260--1262, 1982.

\bibitem{ENYI07}
B.~Engquist and L.~Ying.
\newblock Fast directional multilevel algorithms for oscillatory kernels.
\newblock {\em SIAM J. Sci. Comput.}, 29(4):1710--1737, 2007.

\bibitem{ERSA98}
S.~Erichsen and S.~A. Sauter.
\newblock Efficient automatic quadrature in 3-d {G}alerkin {BEM}.
\newblock {\em Comput. Meth. Appl. Mech. Eng.}, 157:215--224, 1998.

\bibitem{GOVL96}
G.~H. Golub and C.~F. {Van~Loan}.
\newblock {\em {M}atrix {C}omputations}.
\newblock Johns Hopkins University Press, London, 1996.

\bibitem{GRHA02}
L.~Grasedyck and W.~Hackbusch.
\newblock Construction and arithmetics of {${\mathcal{H}}$}-matrices.
\newblock {\em Computing}, 70:295--334, 2003.

\bibitem{GRHUROWA98}
L.~Greengard, J.~Huang, V.~Rokhlin, and S.~Wandzura.
\newblock Accelerating fast multipole methods for the {Helmholtz} equation at
  low frequencies.
\newblock {\em IEEE Comp. Sci. Eng.}, 5(3):32--38, 1998.

\bibitem{GRRO87}
L.~Greengard and V.~Rokhlin.
\newblock A fast algorithm for particle simulations.
\newblock {\em J. Comp. Phys.}, 73:325--348, 1987.

\bibitem{HA99}
W.~Hackbusch.
\newblock A sparse matrix arithmetic based on $\mathcal{H}$-matrices. {P}art
  {I}: {I}ntroduction to $\mathcal{H}$-matrices.
\newblock {\em Computing}, 62(2):89--108, 1999.

\bibitem{HAKH00}
W.~Hackbusch and B.~N. Khoromskij.
\newblock A sparse matrix arithmetic based on $\mathcal{H}$-matrices. {P}art
  {II}: {A}pplication to multi-dimensional problems.
\newblock {\em Computing}, 64:21--47, 2000.

\bibitem{HANO89}
W.~Hackbusch and Z.~P. Nowak.
\newblock On the fast matrix multiplication in the boundary element method by
  panel clustering.
\newblock {\em Numer. Math.}, 54(4):463--491, 1989.

\bibitem{MESCDA12}
M.~Messner, M.~Schanz, and E.~Darve.
\newblock Fast directional multilevel summation for oscillatory kernels based
  on {Chebyshev} interpolation.
\newblock {\em J. Comp. Phys.}, 231(4):1175--1196, 2012.

\bibitem{BOMI96}
E.~Michielssen and A.~Boag.
\newblock A multilevel matrix decomposition algorithm for analyzing scattering
  from large structures.
\newblock {\em IEEE Trans. Antennas and Propagation}, AP-44:1086--1093, 1996.

\bibitem{RO85}
V.~Rokhlin.
\newblock Rapid solution of integral equations of classical potential theory.
\newblock {\em J. Comp. Phys.}, 60:187--207, 1985.

\bibitem{RO93}
V.~Rokhlin.
\newblock Diagonal forms of translation operators for the {Helmholtz} equation
  in three dimensions.
\newblock {\em Appl. Comp. Harm. Anal.}, 1:82--93, 1993.

\bibitem{SA00}
S.~A. Sauter.
\newblock Variable order panel clustering.
\newblock {\em Computing}, 64:223--261, 2000.

\bibitem{SASC11}
S.~A. Sauter and C.~Schwab.
\newblock {\em Boundary Element Methods}.
\newblock Springer, 2011.

\end{thebibliography}

\end{document}